\documentclass[11pt, reqno, psamsfonts]{amsart}
\pdfoutput=1

\usepackage{amssymb}
\usepackage{amsthm}
\usepackage{amsmath}
\usepackage{latexsym}
\usepackage[T1]{fontenc}
\usepackage[utf8]{inputenc}
\usepackage[russian, french, english]{babel}

\usepackage{graphicx}
\usepackage{wrapfig}
\usepackage[justification=centering, labelfont=bf]{caption}
\usepackage{mathtools}
\usepackage{tikz}
\usepackage{subcaption}
\usepackage{amsbsy}
\usepackage[inline]{enumitem}
\usepackage{mathrsfs}
\usetikzlibrary{shapes,snakes}
\usetikzlibrary{arrows.meta}
\usetikzlibrary{decorations.pathmorphing}
\usetikzlibrary{patterns}
\usepackage{float}
\usepackage{array}
\usepackage{multicol}
\usepackage{stmaryrd}
\usepackage{cancel} 
\usepackage{lmodern}
\usepackage{upgreek}
\usepackage{titlesec}
\usepackage{bbm}
\usepackage[spacing=true,kerning=true,babel=true,tracking=true]{microtype}
\usepackage[shortcuts]{extdash}
\usepackage[foot]{amsaddr}
\usepackage[left=1in,right=1in,top=0.95in,bottom=0.95in,bindingoffset=0cm]{geometry}
\usepackage{bm}
\usepackage{centernot}
\usepackage{mdframed}
\usepackage{multirow}
\usepackage{hyperref}
\hypersetup{
    colorlinks=true,
    linkcolor=blue,
    filecolor=magenta,
    urlcolor=gray,
    citecolor=gray,
}
\allowdisplaybreaks

\usepackage{nameref}

\makeatletter
\let\orgdescriptionlabel\descriptionlabel
\renewcommand*{\descriptionlabel}[1]{%
  \let\orglabel\label
  \let\label\@gobble
  \phantomsection
  \edef\@currentlabel{#1}%
  \let\label\orglabel
  \orgdescriptionlabel{#1}%
}
\makeatother

\usepackage[backend=biber, style=alphabetic, sorting=nyt, maxnames=100,backref=true,sortcites = true]{biblatex}
\title{Coloring graphs with forbidden \\ almost bipartite subgraphs}
\date{}
\author{\lsstyle James~Anderson}
\email{james.anderson@math.gatech.edu}
\author{\lsstyle Anton~Bernshteyn}
\email{bernshteyn@math.ucla.edu}
\author{\lsstyle Abhishek~Dhawan}
\email{adhawan2@illinois.edu}
\address{{\normalfont{(JA) School of Mathematics, Georgia Institute of Technology, Atlanta, GA, USA}}}
\address{\normalfont{(AB) Department of Mathematics, University of California, Los Angeles, CA, USA}}
\address{\normalfont{(AD) Department of Mathematics, University of Illinois Urbana--Champaign, IL, USA}}
\thanks{This research was partially supported by the NSF grant DMS-2045412 and the NSF CAREER grant DMS-2528522.
AD's research was also partially supported by the NSF RTG grant DMS-1937241.}

\newtheoremstyle{bfnote}%
{}{}%
{\slshape}{}%
{\bfseries}{\bfseries.}%
{ }%
{\thmname{#1}\thmnumber{ #2}\thmnote{ \ep{\normalfont{}#3}}}

\newtheoremstyle{claim}%
{}{}%
{\slshape}{}%
{\itshape}{.}%
{ }%
{\thmname{#1}\thmnumber{ #2}\thmnote{ \ep{\normalfont{}#3}}}

\theoremstyle{bfnote}
\newtheorem{theo}{Theorem}[section]
\newtheorem*{theo*}{Theorem}
\newtheorem*{Lemma*}{Lemma}
\newtheorem{prop}[theo]{Proposition}
\newtheorem{Lemma}[theo]{Lemma}

\newtheorem{corl}[theo]{Corollary}
\newtheorem{conj}[theo]{Conjecture}

\newtheorem*{corl*}{Corollary}

\theoremstyle{definition}
\newtheorem{defn}[theo]{Definition}
\newtheorem*{defn*}{Definition}

\newtheorem{ques}[theo]{Question}
\newtheorem{prob}[theo]{Problem}

\newtheorem*{exmp*}{Example}

\theoremstyle{remark}
\newtheorem*{ques*}{Question}
\newtheorem*{remk*}{Remark}

\theoremstyle{claim}

\newcounter{ForClaims}[section]
\newtheorem{claim}{Claim}[ForClaims]

\newtheorem*{claim*}{Claim}

\makeatletter
\newcommand{\neutralize}[1]{\expandafter\let\csname c@#1\endcsname\count@}
\makeatother

\newenvironment{claimproof}{\noindent$\rhd$\hspace{1em}}{\hfill$\blacktriangleleft$\smallskip}

\newcommand{\0}{\emptyset}
\newcommand{\set}[1]{\{#1\}}
\newcommand{\N}{{\mathbb{N}}}

\renewcommand{\P}{\mathbb{P}}
\newcommand{\E}{\mathbb{E}}
\renewcommand{\epsilon}{\varepsilon}

\renewcommand{\phi}{\varphi}
\renewcommand{\theta}{\vartheta}
\renewcommand{\leq}{\leqslant}
\renewcommand{\geq}{\geqslant}
\newcommand{\defeq}{\coloneqq}
\newcommand{\im}{\mathsf{im}}
\newcommand{\bemph}[1]{{\normalfont#1}} 
\newcommand{\ep}[1]{\bemph{(}#1\bemph{)}} 

\newcommand{\pto}{\dashrightarrow}
\newcommand{\emphdef}[1]{\textbf{\textit{{#1}}}}
\newcommand{\keep}{\mathsf{keep}}
\newcommand{\uncolor}{\mathsf{uncolor}}
\newcommand{\blank}{\mathsf{blank}}

\newcommand{\dom}{\mathsf{dom}}
\newcommand{\eq}{\mathsf{eq}}
\newcommand{\Ber}{\mathsf{Bernoulli}}
\newcommand{\LLL}{\text{Lov\'asz Local Lemma}}
\numberwithin{equation}{section}
\newcommand{\emphd}[1]{\emphdef{#1}}
\newcommand{\Bad}{\mathsf{Bad}}
\newcommand{\Good}{\mathsf{Good}}
\newcommand{\keptedges}{E_{K}(v)}
\newcommand{\uncoloredges}{E_{U}(v)}
\newcommand{\normalizeddeg}{d_{K\cap U}(v)}
\newcommand{\nd}{\normalizeddeg}
\newcommand{\uncoloredgesc}{E_{U}(v,c)}

\titleformat{\section}[block]{\scshape}{\thesection.}{1ex}{}
\titleformat{\subsection}[block]{\bfseries}{\thesubsection.}{1ex}{}
\titleformat{\subsection}[block]{\bfseries}{\thesubsection.}{1ex}{}
\titleformat{\subsubsection}[runin]{\itshape}{\bfseries\upshape\thesubsubsection.}{1ex}{}[.---]

\titlespacing*{\section}{0pt}{*3}{*1}
\titlespacing*{\subsection}{0pt}{*3}{*1}
\titlespacing*{\subsubsection}{0pt}{*1.5}{*0}

\renewbibmacro{in:}{}

\renewbibmacro*{volume+number+eid}{%
	\printfield{volume}%
	\setunit*{\addnbspace}
	\printfield{number}%
	\setunit{\addcomma\space}%
	\printfield{eid}}

\DeclareFieldFormat[article]{volume}{\textbf{#1}\space}
\DeclareFieldFormat[article]{number}{\mkbibparens{#1}}

\DeclareFieldFormat{journaltitle}{#1,}
\DeclareFieldFormat[thesis]{title}{\mkbibemph{#1}\addperiod}
\DeclareFieldFormat[article, unpublished, thesis]{title}{\mkbibemph{#1},}
\DeclareFieldFormat[book]{title}{\mkbibemph{#1}\addperiod}
\DeclareFieldFormat[unpublished]{howpublished}{#1, }

\DeclareFieldFormat{pages}{#1}

\DeclareFieldFormat[article]{series}{Ser.~#1\addcomma}

\setlist{topsep=4pt,itemsep=4pt}

\begin{document}


\maketitle


\begin{abstract}
    Alon, Krivelevich, and Sudakov conjectured in 1999 that for every finite graph $F$, there exists a quantity $c(F)$ such that $\chi(G) \leq (c(F) + o(1)) \Delta / \log\Delta$ whenever $G$ is an $F$-free graph of maximum degree $\Delta$. The largest class of connected graphs $F$ for which this conjecture has been verified so far, by Alon, Krivelevich, and Sudakov themselves, comprises the almost bipartite graphs (i.e., subgraphs of the complete tripartite graph $K_{1,t,t}$ for some $t \in \N$). However, the optimal value for $c(F)$ remains unknown even for such graphs. Bollob\'as showed, using random regular graphs, that $c(F) \geq 1/2$ when $F$ contains a cycle. On the other hand, Davies, Kang, Pirot, and Sereni recently established an upper bound of $c(K_{1,t,t}) \leq t$. We improve this to a uniform constant, showing $c(F) \leq 4$ for every almost bipartite graph $F$. This surprisingly makes the bound independent of $F$ in all the known cases of the conjecture. We also establish a more general version of our bound in the setting of DP-coloring (also known as correspondence coloring) and consider some algorithmic consequences of our results.
\end{abstract}

\section{Introduction}

    \subsection{Bounds for the Alon--Krivelevich--Sudakov conjecture}

    All graphs in this paper are finite, undirected, and simple. Unless explicitly indicated otherwise, all logarithms are base $e$. We say that a graph $G$ is \emphdef{$F$-free} if $G$ has no subgraph \ep{not necessarily induced} isomorphic to $F$. We are interested in the following question:
    
    \begin{ques}\label{ques:AKS}
        Given a graph $F$ and a natural number $\Delta$, how large can the chromatic number $\chi(G)$ of an $F$-free graph $G$ of maximum degree $\Delta$ be?
    \end{ques}
    
    The best known general upper bound for Question~\ref{ques:AKS}, due to Johansson \cite{Joh_sparse}, is
    \begin{equation}\label{eq:loglog}
        \chi(G) \,=\, O \left(\frac{\Delta \log\log \Delta}{\log \Delta}\right).
    \end{equation}
    A simple proof of this result was recently discovered by Molloy \cite{Molloy} \ep{see also \cite[\S4]{JMTheorem} and \cite[\S7.2]{DKPS}}. 
    Alon, Krivelevich, and Sudakov famously conjectured that 
    the $\log \log \Delta$ factor in \eqref{eq:loglog} can be removed:
    
    \begin{conj}[{Alon--Krivelevich--Sudakov \cite[Conjecture 3.1]{AKSConjecture}}]\label{conj:AKS}
        Fix an arbitrary graph $F$. If $G$ is an $F$-free graph of maximum degree $\Delta$, then $\chi(G) = O(\Delta/\log\Delta)$ \ep{where the $O(\cdot)$ notation hides factors that may depend on $F$}.
    \end{conj}

    Conjecture~\ref{conj:AKS} is a strengthening of an earlier conjecture of Ajtai, Erd\H{o}s, Koml\'os, and Szemer\'edi \cite{AEKS}, who asked whether every $n$-vertex $F$-free graph $G$ of maximum degree $\Delta$ has independence number $\alpha(G) = \Omega(n \log \Delta / \Delta)$. Both conjectures remain wide open. The goal of this paper is to investigate the supporting evidence for Conjecture \ref{conj:AKS} in greater depth, and our main contributions---Theorems~\ref{theo:col} and \ref{mainTheorem}---improve and generalize most of the known results on the subject.
    
    To refine the discussion, we introduce the following notation. Given a graph $F$ and a natural number $\Delta$, let
    \[
        c(F, \Delta) \defeq \max \left\{ \frac{\chi(G)}{\Delta/\log \Delta} \,:\, \text{$G$ is an $F$-free graph of maximum degree $\Delta$} \right\},
    \]
    and set $c(F) \defeq \limsup_{\Delta \to \infty} c(F, \Delta)$. In other words, if $c(F)$ is finite, then it is the minimum quantity such that \[\chi(G) \leq (c(F) + o(1))\frac{\Delta}{\log \Delta}\] for all $F$-free graphs $G$ of maximum degree $\Delta$, where $o(1)$ indicates a function of $\Delta$ that approaches $0$ as $\Delta \to \infty$. Conjecture~\ref{conj:AKS} asserts that $c(F)$ is finite for every graph $F$.
    
    It is easy to see that Conjecture~\ref{conj:AKS} holds for a graph $F$ if and only if it holds for every connected component of $F$. Thus, we shall be primarily concerned with connected graphs $F$ in this paper. To date, the largest class of connected graphs for which Conjecture~\ref{conj:AKS} has been verified comprises the so-called \emphd{almost bipartite} graphs, i.e., graphs that can be made bipartite by removing at most one vertex. Equivalently, a graph is almost bipartite if it is isomorphic to a subgraph of the complete tripartite graph $K_{1,t,t}$ for some $t \in \N$. Conjecture~\ref{conj:AKS} was proved for almost bipartite graphs by Alon, Krivelevich, and Sudakov:
    
    \begin{theo}[{Alon--Krivelevich--Sudakov \cite[Corollary 2.4]{AKSConjecture}}]\label{theo:AKSAlmostBipartite}
        Conjecture~\ref{conj:AKS} holds for almost bipartite graphs $F$. That is, $c(F) < \infty$ for every almost bipartite graph $F$.
    \end{theo}
    
    In their proof of Theorem~\ref{theo:AKSAlmostBipartite}, Alon, Krivelevich, and Sudakov observe that a graph $G$ is $K_{1,t,t}$-free if and only if the neighborhood of every vertex in $G$ is $K_{t,t}$-free. By the \hyperref[KST]{K\H{o}v\'ari--Sós--Turán theorem} \cite{KST},
    this implies that the neighborhood of every vertex in $G$ spans at most $O(\Delta^{2-1/t})$ edges. Therefore, Theorem~\ref{theo:AKSAlmostBipartite} is a consequence of the following more general result:
    
    \begin{theo}[{Alon--Krivelevich--Sudakov \cite[Theorem 1.1]{AKSConjecture}}]\label{theo:sparse}
        If $G$ is a graph of maximum degree $\Delta$ where the neighborhood of every vertex spans at most $\Delta^2/k$ edges, then $\chi(G) = O(\Delta/\log k)$.
    \end{theo}
    
    Applying Theorem~\ref{theo:sparse} with $k = \Omega(\Delta^{1/t})$ yields the bound $c(K_{1,t,t}) = O(t)$. This was recently sharpened to $c(K_{1,t,t}) \leq t$ by Davies, Kang, Pirot, and Sereni:
    
    \begin{theo}[{Davies--Kang--Pirot--Sereni \cite[\S5.6]{DKPS}}]\label{theo:DKPSAlmostBipartite}
        For every $t \in \N$, $c(K_{1,t,t}) \leq t$.
    \end{theo}
    
    Theorem~\ref{theo:DKPSAlmostBipartite} follows from a stronger form of Theorem~\ref{theo:sparse} due to Davies, Kang, Pirot, and Sereni, namely the bound $\chi(G) \leq (1 + o(1))\Delta/\log k$ for graphs in which every neighborhood spans at most $\Delta^2/k$ edges \cite[Theorem~5]{DKPS}. Here we strengthen Theorem~\ref{theo:DKPSAlmostBipartite} further and establish a \emph{uniform} upper bound on $c(F)$ for all almost bipartite $F$:
    
    \begin{theo}\label{theo:col}
        For every almost bipartite graph $F$, $c(F) \leq 4$.
    \end{theo}
    
    The fact that we can bound $c(F)$ by a universal constant in all the known cases of Conjecture~\ref{conj:AKS} is somewhat surprising, especially since, 
    by another result of Alon, Krivelevich, and Sudakov, the bound in Theorem~\ref{theo:sparse} is tight up to a constant factor \cite[Proposition 1.2]{AKSConjecture}. In particular, in contrast to Theorem~\ref{theo:DKPSAlmostBipartite}, Theorem~\ref{theo:col} cannot be derived from a stronger version of Theorem~\ref{theo:sparse}, i.e., a different approach is required to prove it. It is also worth emphasizing that the following lower bound on the independence number of $F$-free graphs, which follows immediately from Theorem~\ref{theo:col}, is already new:

    \begin{corl}\label{corl:alpha}
        Fix an almost bipartite graph $F$. If $G$ is an $n$-vertex $F$-free graph of maximum degree $\Delta$, then
        \[
            \alpha(G) \,\geq\, \left(\frac{1}{4} - o(1)\right) \frac{\log \Delta}{\Delta} \,n.
        \]
    \end{corl}
    
    Prior to this work, it was not known whether a lower bound on $\alpha(G)$ could be made independent of $F$ (modulo lower order terms). In particular, for $K_{1,t,t}$-free graphs $G$, the best previously known bound was the one implied by Theorem~\ref{theo:DKPSAlmostBipartite}, i.e.,
    \[
        \alpha(G) \,\geq\, \left(\frac{1}{t} - o(1)\right) \frac{\log \Delta}{\Delta} \,n.
    \]
    
    Let us now revisit Conjecture~\ref{conj:AKS} in light of Theorem~\ref{theo:col}.
    On the one hand, since Theorem~\ref{theo:col} covers \emph{all} connected graphs $F$ currently known to satisfy Conjecture~\ref{conj:AKS}, it seems natural to raise the following question:

    \begin{prob}[Uniform version of the Alon--Krivelevich--Sudakov Conjecture]\label{conj:uniformAKS}
        Is there is a universal constant $c > 0$ such that $c(F) \leq c$ for all graphs $F$?
    \end{prob}

    An affirmative solution to Problem~\ref{conj:uniformAKS} would be a significant breakthrough, as even the weaker Conjecture~\ref{conj:AKS} is far from being settled. On the other hand, the fact that Problem~\ref{conj:uniformAKS} has a positive answer for almost bipartite graphs $F$ 
    may indicate that they form a very special class, and a bold generalization such as Problem~\ref{conj:uniformAKS} (or indeed Conjecture~\ref{conj:AKS}!) could be seen as
    unduly rash. As such, it may be  that Problem~\ref{conj:uniformAKS} is a good target for a \emph{negative} solution, which 
     would serve as a convenient stepping stone toward the more ambitious goal of
    falsifying the original \hyperref[conj:AKS]{Alon--Krivelevich--Sudakov conjecture}. Our own view of the matter can be nicely summed up in the words of Ajtai, Erd\H{o}s, Koml\'os, and Szemer\'edi \cite[314]{AEKS} (they were referring to their conjectured lower bound on the independence number for graphs $G$ under the assumptions of Conjecture~\ref{conj:AKS}):

    \begin{quote}
        ``\textsl{Perhaps {\normalfont[the conjecture]} is too optimistic, but we feel that it is an interesting and challenging question}.''
    \end{quote}

    {
		\renewcommand{\arraystretch}{1.3}
		\begin{table}[b!]
			\caption{Known bounds on $c(F)$.}\label{table:bounds}
			\begin{tabular}{| l || l || l |}
				\hline
				$F$ & $c(F)$ & References\\\hline\hline
				forest & $0$ & \\\hline
				not forest & $\geq 1/2$ & Bollob\'as \cite{BollobasIndependence}
				\\\hline
				\multirow{3}{*}{$K_3$} & finite & Johansson \cite{Joh_triangle} \\\cline{2-3}
				& $\leq 4$ & Pettie--Su \cite{PS15} \\\cline{2-3}
				& $\leq 1$ & Molloy \cite{Molloy}\\\hline
				cycle & $\leq 1$ & Davies--Kang--Pirot--Sereni \cite{DKPS}\\\hline
				fan & $\leq 1$ & Davies--Kang--Pirot--Sereni \cite{DKPS}\\\hline
				bipartite & $\leq 1$ & \cite{BipartiteNibble}\\\hline
				\multirow{3}{*}{$K_{1,t,t}$} & $O(t)$ & Alon--Krivelevich--Sudakov \cite{AKSConjecture} \\\cline{2-3}
				& $\leq t$ & Davies--Kang--Pirot--Sereni \cite{DKPS}\\\cline{2-3}
				& $\leq 4$ & \textbf{this paper}\\\hline
			\end{tabular}
		\end{table}	
	}
    
    To put Theorem~\ref{theo:col} and Problem~\ref{conj:uniformAKS} into perspective, let us now give a brief overview of some of the known bounds on $c(F)$, which we summarize in Table~\ref{table:bounds}.
    If $F$ is a forest, then every $F$-free graph is $(|V(F)| - 2)$-degenerate \cite[Corollary 1.5.4]{Die} and hence its chromatic number is at most $|V(F)| - 1$ \ep{regardless of the maximum degree}, which implies that $c(F) = 0$. On the other hand, if $F$ contains a cycle, then $c(F) \geq 1/2$, because there exist $\Delta$-regular graphs of arbitrarily high girth with chromatic number at least $\Delta/(2\log \Delta)$, as shown by Bollob\'as \cite{BollobasIndependence} using random regular graphs. It is not known if there exist any graphs $F$ with $c(F) > 1/2$. On the other hand, it is also not known if there are any graphs $F$ containing a cycle and satisfying $c(F) < 1$: it turns out that the value $1$ is a natural threshold for upper bounds on $c(F)$, coinciding with the so-called {shattering threshold} for colorings of random graphs of average degree $\Delta$ \cite{Zdeborova,Achlioptas} and the density threshold for factor of i.i.d. independent sets in $\Delta$-regular trees \cite{RV}.

    The first nontrivial case of Conjecture~\ref{conj:AKS}, known a few years before the conjecture was actually made, was $F = K_3$, proved by Johansson \cite{Joh_triangle}. 
    Johansson's result was preceded by the work of Ajtai, Koml\'os, and Szemer\'edi \cite{AjtaiKS1,AjtaiKS2} and Shearer \cite{Shearer} concerning the independence number of triangle-free graphs, and by the bound on the chromatic number for graphs that are both $K_3$- and $C_4$-free due to Kim \cite{Kim95}. 
    Johansson's paper was never published, but a textbook presentation of Johansson's argument can be found in \cite[\S13]{MolloyReed}. Since Johansson's manuscript is unpublished, it is unclear what exact bound on $c(K_3)$ it yields; in \cite[\S13]{MolloyReed}, Molloy and Reed report the bound $c(K_3) \leq 9$ and prove $c(K_3) \leq 160$. This was later improved to $c(K_3) \leq 4$ by Pettie and Su \cite{PS15} and finally to $c(K_3) \leq 1$ by Molloy \cite{Molloy}. An alternative proof of Molloy's bound $c(K_3) \leq 1$ was given in \cite{JMTheorem}. Molloy's result was generalized by Davies, Kang, Pirot, and Sereni, who showed that $c(C_k) \leq 1$ for the $k$-cycle $C_k$ \cite[Theorem 4]{DKPS}, and, even more generally, $c(F_k) \leq 1$ for the $k$-fan $F_k$ \ep{the $k$-fan is obtained from a path on $k - 1$ vertices by adding a new universal vertex
    } \cite[Theorem 6]{DKPS}. In an earlier paper \cite{BipartiteNibble}, we identified another class of graphs $F$ with $c(F) \leq 1$, namely the bipartite graphs.
    
    In view of the above discussion, it is possible that the answer to Problem~\ref{conj:uniformAKS} is ``yes'' with the constant $c$ as small as $1$ or even $1/2$,
    although our current methods do not allow reducing the bound in Theorem~\ref{theo:col} below $4$.
 
    \subsection{DP-coloring}\label{subsec:DP}

    Problems such as Conjecture~\ref{conj:AKS} can naturally be considered for stronger variants of graph coloring, such as list-coloring. This often requires additional insight; for example, the version of Theorem~\ref{theo:sparse} (and hence also of Theorem~\ref{theo:AKSAlmostBipartite}) for list-coloring was proved by Vu \cite{Vu} in a manner quite different from the arguments employed by Alon, Krivelevich, and Sudakov in \cite{AKSConjecture}. Our Theorem~\ref{theo:col} also extends to the list-coloring framework; indeed, we obtain Theorem~\ref{theo:col} as a consequence of a considerably more general result in the context of \emph{DP-coloring} \ep{also known as \emph{correspondence coloring}},  
    introduced by Dvo\v{r}\'ak and Postle \cite{DPCol} as a further 
    generalization of list-coloring. (This notion is closely related to \emph{local conflict coloring} introduced by Fraigniaud, Heinrich, and Kosowski \cite{FHK} with a view toward applications in distributed computing.)
    
    Just as in ordinary list-coloring, we assume that every vertex $v \in V(G)$ of a graph $G$ is given a list $L(v)$ of available colors to choose from. In contrast to list-coloring however, the identifications between the colors in the lists may vary from edge to edge. More precisely, each edge $uv \in E(G)$ is assigned a matching $M_{uv}$ (not necessarily perfect and possibly empty) from $L(u)$ to $L(v)$. A \emph{proper DP-coloring} is a mapping $\phi$ that assigns a color $\phi(v) \in L(v)$ to each vertex $v \in V(G)$ so that whenever $uv \in E(G)$, we have $\phi(u)\phi(v) \notin M_{uv}$. Note that list-coloring is indeed a special case of DP-coloring which occurs when the colors ``correspond to themselves,'' i.e., for each $c \in L(u)$ and $c' \in L(v)$, we have $cc' \in M_{uv}$ if and only if $c = c'$.

    Formally, we describe DP-coloring using an auxiliary graph $H$ called a \emph{DP-cover} of $G$. In this model, we treat the lists of colors assigned to distinct vertices as pairwise disjoint \ep{this is a convenient assumption that does not restrict the generality of the model}. 
    
    \begin{defn}\label{corrcov}
    A \emphdef{DP-cover} \ep{or a \emphd{correspondence cover}} of a graph $G$ is a pair $\mathcal{H} = (L, H)$, where $H$ is a graph and $L \colon V(G) \to 2^{V(H)}$ is a function such that:
    \begin{itemize}
        \item The sets $(L(v) \,:\, v \in V(G))$ partition $V(H)$.
        \item For each $v \in V(G)$, $L(v)$ is an independent set in $H$.
        \item For $u$, $v \in V(G)$, the edges between $L(u)$ and $L(v)$ in $H$ form a matching; this matching is empty whenever $uv \notin E(G)$.
    \end{itemize}
    We call the vertices of $H$ \emphd{colors}.
    For $c \in V(H)$, we let $L^{-1}(c)$ denote the \emphd{underlying vertex} of $c$ in $G$, i.e., the unique vertex $v \in V(G)$ such that $c \in L(v)$. If two colors $c$, $c' \in V(H)$ are adjacent in $H$, we say that they \emphd{correspond} to each other and write $c \sim c'$.
    
    
    An \emphd{$\mathcal{H}$-coloring} is a mapping $\phi \colon V(G) \to V(H)$ such that $\phi(v) \in L(v)$ for all $v \in V(G)$. Similarly, a \emphd{partial $\mathcal{H}$-coloring} is a partial mapping $\phi \colon V(G) \pto V(H)$ such that $\phi(v) \in L(v)$ whenever $\phi(v)$ is defined. A \ep{partial} $\mathcal{H}$-coloring $\phi$ is \emphd{proper} if the image of $\phi$ is an independent set in $H$, i.e., if $\phi(u) \not \sim \phi(v)$ 
    for all $u$, $v \in V(G)$ such that $\phi(u)$ and $\phi(v)$ are both defined.
    
    A DP-cover $\mathcal{H} = (L,H)$ is \emphdef{$k$-fold} if $|L(v)| \geq k$ for all $v \in V(G)$. The \emphdef{DP-chromatic number} of $G$, denoted by $\chi_{DP}(G)$, is the smallest $k$ such that $G$ admits a proper $\mathcal{H}$-coloring with respect to every $k$-fold DP-cover $\mathcal{H}$.
\end{defn}

Since DP-coloring is a generalization of list-coloring, it is clear that $\chi_\ell(G) \leq \chi_{DP}(G)$ for all graphs $G$ (here $\chi_\ell(G)$ denotes the list-chromatic number of $G$). This inequality can be strict; moreover, the ratio $\chi_{DP}(G)/\chi_\ell(G)$ can be arbitrarily large. For example, the complete bipartite graph $K_{t,t}$ with $t \to \infty$ satisfies $\chi(K_{t,t}) = 2$, $\chi_\ell(K_{t,t}) = \Theta(\log t)$, and $\chi_{DP}(K_{t,t}) = \Theta(t/\log t)$ \cite{asymptoticDP}.

    The framework of DP-coloring naturally allows one to put structural constraints on the cover graph $H$ as opposed to the underlying base graph $G$.
    For instance, if $\mathcal{H} = (L, H)$ is a DP-cover of a graph $G$, then $\Delta(H) \leq \Delta(G)$, so an upper bound on $\Delta(H)$ is a weaker assumption than the same upper bound on $\Delta(G)$, and there exist a number of results in list and correspondence coloring in which the number of available colors given to each vertex is a function of $\Delta(H)$ as opposed to $\Delta(G)$. This framework, often referred to as the \emph{color-degree setting}, was pioneered by Kahn \cite{KahnListEdge}, Kim \cite{Kim95}, Johansson \cite{Joh_triangle,Joh_sparse}, and Reed \cite{Reed}, among others. For a selection of a few more recent examples, see \cite{BohmanHolzman, ReedSud, LohSudakov, Palette, CK, KangKelly, GlockSudakov}.
    
    In general, $G$ being $F$-free does not imply that a cover graph $H$ is $F$-free as well (see Fig.~\ref{fig:not free}), but it does so for certain graphs $F$:

    \begin{prop}[{\cite[Proposition 1.10(S2)]{anderson2024coloring}}]\label{prop: G F free implies H F free}
        Let $F$ and $G$ be graphs and let $\mathcal{H} = (L, H)$ be a DP-cover of $G$.
        Assume that every two nonadjacent vertices in $F$ have a common neighbor. If $G$ is $F$-free, then so is $H$.
    \end{prop}

    Note that if Conjecture~\ref{conj:AKS} holds for complete graphs $F$, then it holds for all graphs (because every graph is a subgraph of a complete graph). Since complete graphs $F$ vacuously satisfy the assumption in Proposition~\ref{prop: G F free implies H F free}, the following statement is a natural strengthening of Conjecture~\ref{conj:AKS}:

    \begin{figure}[tb!]
	\centering
	\begin{subfigure}[t]{.43\textwidth}
		\centering
		\begin{tikzpicture}[scale=2]
		\node[circle,fill=black,draw,inner sep=0pt,minimum size=4pt] (a) at (0,0) {};
		\path (a) ++(90:1) node[circle,fill=black,draw,inner sep=0pt,minimum size=4pt] (b) {};
            \path (a) ++(30:1) node[circle,fill=black,draw,inner sep=0pt,minimum size=4pt] (c) {};
            \path (c) ++(30:1) node[circle,fill=black,draw,inner sep=0pt,minimum size=4pt] (d) {};
            \path (d) ++(-90:1) node[circle,fill=black,draw,inner sep=0pt,minimum size=4pt] (e) {};

            \draw[thick] (c) -- (a) -- (b) -- (c) -- (d) -- (e) -- (c);
		
		\end{tikzpicture}
		\caption{A $C_6$-free graph $G$.}\label{fig:not free:graph}
	\end{subfigure}%
	\qquad\qquad%
	\begin{subfigure}[t]{.43\textwidth}
		\centering
		\begin{tikzpicture}[scale=2]
		\node[circle,fill=black,draw,inner sep=0pt,minimum size=4pt] (a) at (0,0) {};
		\path (a) ++(90:1) node[circle,fill=black,draw,inner sep=0pt,minimum size=4pt] (b) {};
            \path (a) ++(30:1) node[circle,fill=black,draw,inner sep=0pt,minimum size=4pt] (c) {};
            \path (c) ++(30:1) node[circle,fill=black,draw,inner sep=0pt,minimum size=4pt] (d) {};
            \path (d) ++(-90:1) node[circle,fill=black,draw,inner sep=0pt,minimum size=4pt] (e) {};

            \path (a) ++(90:0.3) node[circle,fill=black,draw,inner sep=0pt,minimum size=4pt] (a1) {};
            \path (b) ++(90:0.3) node[circle,fill=black,draw,inner sep=0pt,minimum size=4pt] (b1) {};
            \path (c) ++(90:0.3) node[circle,fill=black,draw,inner sep=0pt,minimum size=4pt] (c1) {};
            \path (d) ++(90:0.3) node[circle,fill=black,draw,inner sep=0pt,minimum size=4pt] (d1) {};
            \path (e) ++(90:0.3) node[circle,fill=black,draw,inner sep=0pt,minimum size=4pt] (e1) {};

            \draw[rounded corners=2pt, black] ([xshift=-6pt, yshift=8pt]a1.south) -- ([xshift=-6pt, yshift=-8pt]a.north) -- ([xshift=6pt, yshift=-8pt]a.north) -- ([xshift=6pt, yshift=8pt]a1.south) -- cycle;

            \draw[rounded corners=2pt, black] ([xshift=-6pt, yshift=8pt]b1.south) -- ([xshift=-6pt, yshift=-8pt]b.north) -- ([xshift=6pt, yshift=-8pt]b.north) -- ([xshift=6pt, yshift=8pt]b1.south) -- cycle;

            \draw[rounded corners=2pt, black] ([xshift=-6pt, yshift=8pt]c1.south) -- ([xshift=-6pt, yshift=-8pt]c.north) -- ([xshift=6pt, yshift=-8pt]c.north) -- ([xshift=6pt, yshift=8pt]c1.south) -- cycle;

            \draw[rounded corners=2pt, black] ([xshift=-6pt, yshift=8pt]d1.south) -- ([xshift=-6pt, yshift=-8pt]d.north) -- ([xshift=6pt, yshift=-8pt]d.north) -- ([xshift=6pt, yshift=8pt]d1.south) -- cycle;

            \draw[rounded corners=2pt, black] ([xshift=-6pt, yshift=8pt]e1.south) -- ([xshift=-6pt, yshift=-8pt]e.north) -- ([xshift=6pt, yshift=-8pt]e.north) -- ([xshift=6pt, yshift=8pt]e1.south) -- cycle;

            \draw[thick] (a1) -- (c) -- (e1) -- (d) -- (c1) -- (b) -- (a1);

		\end{tikzpicture}
		\caption{A $2$-fold cover of $G$ containing a $C_6$.}\label{fig:not free:cover}
	\end{subfigure}%
	\caption{A $C_6$-free graph with a $2$-fold cover containing $C_6$.}\label{fig:not free}
\end{figure}

    \begin{conj}[DP-version of the Alon--Krivelevich--Sudakov Conjecture]\label{conj:ABD}
        For every graph $F$, there exist constants $c$, $d_0 > 0$ such that the following holds. Let $G$ be a graph and let $\mathcal{H} = (L,H)$ be a DP-cover of $G$. If $H$ is $F$-free and has maximum degree $d \geq d_0$ and if $|L(v)| \geq c d /\log d$ for all $v \in V(G)$, then $G$ admits a proper $\mathcal{H}$-coloring.
    \end{conj}
    
    In \cite{BipartiteNibble} we confirmed Conjecture~\ref{conj:ABD} in the case when $F$ is bipartite with $c = 1 + o(1)$. The main result of this paper is a proof of Conjecture~\ref{conj:ABD} for almost bipartite $F$ with $c = 4 + o(1)$:
    
    \begin{theo}\label{mainTheorem}
        There is a constant $\alpha > 0$ such that for every $\epsilon > 0$, there is $d_0 \in \N$ such that the following holds. Suppose that $d$, $s$, $t \in \N$ satisfy
        \[
            d \geq d_0,\quad s \leq d^{\alpha\epsilon}, \quad \text{and} \quad  t \leq \frac{\alpha\epsilon\log d}{\log \log d}.
        \]
        If $G$ is a graph and $\mathcal{H} = (L,H)$ is a DP-cover of $G$ such that:
        \begin{enumerate}[label=\ep{\normalfont\roman*}]
            \item $H$ is $K_{1,s,t}$-free,
            \item $\Delta(H) \leq d$, and
            \item $|L(v)| \geq (4+\epsilon)d/\log d$ for all $v \in V(G)$,
        \end{enumerate}
        then $G$ has a proper $\mathcal{H}$-coloring.
    \end{theo}
    
    By treating $s$ and $t$ as constants, we obtain the following immediate corollary:
    
    \begin{corl}\label{mainCorollary}
        For every $\epsilon > 0$ and an almost bipartite graph $F$, there is $d_0 \in \N$ such that the following holds. Let $d \geq d_0$. Suppose $\mathcal{H} = (L,H)$ is a DP-cover of $G$ such that:
        \begin{enumerate}[label=\ep{\normalfont\roman*}]
            \item $H$ is $F$-free,
            \item $\Delta(H) \leq d$, and
            \item $|L(v)| \geq (4+\epsilon)d/\log d$ for all $v \in V(G)$.
        \end{enumerate}
        Then $G$ has a proper $\mathcal{H}$-coloring.
    \end{corl}

    Note that if $F$ is an almost bipartite graph and $G$ is $F$-free, then $G$ is $K_{1,t,t}$-free for $t \defeq |V(F)|$. Since every two nonadjacent vertices in $K_{1,t,t}$ have a common neighbor, every DP-cover of $G$ is also $K_{1,t,t}$-free by Proposition~\ref{prop: G F free implies H F free}. Therefore, applying Corollary~\ref{mainCorollary} with $d \defeq \Delta(G)$ and $K_{1,t,t}$ in place of $F$ yields a version of Theorem~\ref{theo:col} for DP-chromatic numbers:
    \begin{corl}\label{corl:DP}
        For every $\epsilon > 0$ and an almost bipartite graph $F$, there is $\Delta_0 \in \N$ such that every $F$-free graph $G$ with maximum degree $\Delta \geq \Delta_0$ satisfies $\chi_{DP}(G) \leq (4 + \epsilon)\Delta/\log \Delta$. 
    \end{corl}
    
    It would be interesting to know when the constant factor in the above results can be reduced from $4 + o(1)$ to $1 + o(1)$. For example, the following conjecture is open:
    
    \begin{conj}[{Cambie--Kang \cite[Conjecture 4]{CK}}]\label{conj:CK}
        For every $\epsilon > 0$, there is $d_0 \in \N$ such that the following holds. Let $G$ be a graph and let $\mathcal{H} = (L,H)$ be a DP-cover of $G$. Suppose that $H$ has maximum degree $d \geq d_0$ and $|L(v)| \geq (1+\epsilon)d/\log d$ for all $v \in V(G)$. Assume also that:
        \begin{itemize}
            \item \ep{weak version} $G$ is triangle-free, or
            \item \ep{strong version} $H$ is triangle-free.
        \end{itemize}
        Then $G$ admits a proper $\mathcal{H}$-coloring.
    \end{conj}
    
    Conjecture~\ref{conj:CK} holds if $d$ is taken to be the maximum degree of $G$ rather than of $H$ \ep{the weak version is proved in \cite{JMTheorem} by the second author and the strong one in \cite{Prz} by Przyby\l{}o}. Cambie and Kang proved the conclusion of Conjecture~\ref{conj:CK} when $G$ is not just triangle-free but bipartite \cite[Corollary 3]{CK}. Amini and Reed \cite{AminiReed} and, independently, Alon and Assadi \cite[Proposition 3.2]{Palette} verified the weak version of Conjecture~\ref{conj:CK} in the list-coloring setting, but with $1 + o(1)$ replaced by a larger constant \ep{$8$ in \cite{Palette}}. Since $K_3$ is almost bipartite, Corollary~\ref{mainCorollary} implies the strong version of Conjecture \ref{conj:CK} with $1 + o(1)$ replaced with $4+o(1)$. Reducing the constant factor to $1 + o(1)$ remains a tantalizing open problem, even in the list-coloring framework.

    Another direction in which it may be possible to generalize our results is to further relax the conditions in Theorem~\ref{mainTheorem} by replacing the maximum degree $\Delta(H)$ of $H$ by the following parameter, called the \emph{maximum average color-degree} of $\mathcal{H}$:
    \[
        \overline{\Delta}(\mathcal{H}) \,\defeq\, \max_{v \in V(G)} \frac{1}{|L(v)|}\sum_{c \in L(v)}\deg_{H}(c).
    \]
    This quantity has been studied in the context of list- and DP-coloring by Kang and Kelly \cite{KangKelly} and Glock and Sudakov \cite{GlockSudakov}, and implicitly appeared in the earlier work of Molloy and Thron \cite{Adaptive} and Dvo\v{r}\'ak, Esperet, Kang, and Ozeki \cite{SingleConflict} (see \cite[\S1.1]{KangKelly}). Moreover, as we explain in \S\ref{subsec:sketch}, average color-degree plays a crucial role in our proof approach. Unfortunately, our technique still requires an upper bound on $\Delta(H)$, not just $\overline{\Delta}(\mathcal{H})$; it remains an interesting open problem to determine whether the dependence on $\Delta(H)$ can be eliminated without any loss in the constant factor.

    \subsection{Algorithmic considerations}\label{subsec:alg}
    
    We establish our main results by iteratively applying the \hyperref[LLL]{Lov\'asz Local Lemma}. Thanks to the algorithmic version of the Lov\'asz Local Lemma developed by Moser and Tardos \cite{MT}, it is routine to verify that our arguments yield efficient randomized coloring algorithms. For example, given an almost bipartite graph $F$, we obtain a randomized algorithm for $(4+o(1))\Delta/\log\Delta$-coloring $F$-free graphs $G$ of maximum degree $\Delta$ that runs in time $\mathrm{poly}(\Delta)n$.
    
    Our work also has some relevance for the theory of \emph{sublinear algorithms}, i.e., algorithms whose computational resources are substantially smaller than their input size. In \cite{Sparse}, Assadi, Chen, and Khanna introduced a general approach for constructing such algorithms based on a technique they called \emph{palette sparsification}. The idea of their method, in a nutshell, is as follows. Suppose we are trying to properly color a graph $G$ using $\ell$ colors, but, due to limited computational resources, we cannot keep track of all the edges of $G$ at once. Let us independently sample, for each vertex $v \in V(G)$, a random set $S(v)$ of colors of size $|S(v)| = s \ll \ell$. Define
    \[
        E_\mathrm{conflict} \,\defeq\, \set{uv \in E(G) \,:\, S(u) \cap S(v) \neq \0}.
    \]
    If we color $G$ by assigning to every vertex $v \in V(G)$ a color from the corresponding set $S(v)$, then only the edges in $E_\mathrm{conflict}$ may become monochromatic, so instead of working with the entire edge set of $G$, we only need to keep track of the edges in the \ep{potentially much smaller} set $E_\mathrm{conflict}$. For this strategy to succeed, we must ensure that, with high probability, it is indeed possible to properly color $G$ using the colors from the sets $S(v)$. Such a result is referred to as a \emph{palette sparsification theorem}. Assadi, Chen, and Khanna \cite{Sparse} demonstrated that palette sparsification theorems can be used in a black-box manner to develop algorithms in a variety of classical models of sublinear computation, including \ep{dynamic} graph streaming algorithms, sublinear time/query algorithms, and massively parallel computation \ep{MPC} algorithms. 
    
    These algorithmic applications motivate the search for new palette sparsification results. In their seminal work \cite{Sparse}, Assadi, Chen, and Khanna proved a palette sparsification theorem for $(\Delta + 1)$-colorings. More precisely, they showed that if $G$ is an $n$-vertex graph of maximum degree $\Delta$, then, with high probability, after sampling $\Theta(\log n)$ colors per vertex independently from a set of $\Delta + 1$ colors, it is possible to find a proper coloring of $G$ from the sampled lists. In \cite{KahnKenney}, Kahn and Kenney tightened this result by showing that sampling $(1+o(1))\log n$ colors per vertex is enough (which is best possible). The literature on the subject is quite extensive, especially since notions closely related to palette sparsification have also been studied independently of their algorithmic applications, with a particular attention in recent years directed to the case of list-\emph{edge}-coloring from random lists. For a sample of results concerning palette sparsification from the algorithmic perspective, see \cite{Sparsification1, Palette, Sparsification2, Sparsification3}; for a more purely probabilistic focus, see, e.g., \cite{KN1,KN2,C1,C2,C3,CH,C4,rand4,rand3,rand1,rand2}.
    
    The paper \cite{Palette} by Alon and Assadi is particularly relevant to our present work. Alongside several other palette sparsification theorems, it includes the following result for $O(\Delta/\log\Delta)$-colorings of triangle-free graphs:
    
    \begin{theo}[{Alon--Assadi \cite[Theorem 2]{Palette}}]\label{sparseAA}
        There exists a constant $C > 0$ such that, for all $0 < \gamma < 1$, there is $\Delta_0 \in \N$ with the following property. Let $G$ be a triangle-free $n$-vertex graph of maximum degree $\Delta \geq \Delta_0$. Let $\ell \geq s$ be positive integers such that
        \[
            \ell \,=\, \frac{9}{\gamma} \,\frac{\Delta}{\log \Delta} \qquad \text{and} \qquad s \,\geq\, \Delta^\gamma + C\sqrt{\log n}.
        \]
        Independently for each $v \in V(G)$, pick a uniformly random subset $S(v) \subseteq [\ell]$ of size $s$. Then, with probability at least $1 - 1/n$, $G$ has a proper coloring $\phi$ with $\phi(v) \in S(v)$ for all $v \in V(G)$.
    \end{theo}
    
    Corollary~\ref{mainCorollary} yields an extension of Theorem~\ref{sparseAA} to $F$-free graphs for any almost bipartite graph $F$; moreover, this extension holds in the DP-coloring setting:
    
    \begin{corl}\label{corl:sparse}
        For every $0 < \epsilon < 1$ and an almost bipartite graph $F$, there exists $C > 0$ such that, for all $0 < \gamma < 1$, there is $d_0 \in \N$ with the following property. Let $d \geq d_0$, $n \in \N$ and suppose that positive integers $\ell \geq s$ satisfy
        \[
            \ell \,=\, \frac{4+\epsilon}{\gamma} \,\frac{d}{\log d} \qquad \text{and} \qquad s \,\geq\, d^\gamma + C\sqrt{\log n}.
        \]
        Let $\mathcal{H} = (L,H)$ be a DP-cover of an $n$-vertex graph $G$ such that:
        \begin{enumerate}[label=\ep{\normalfont\roman*}]
            \item $H$ is $F$-free,
            \item $\Delta(H) \leq d$, and
            \item $|L(v)| = \ell$ for all $v \in V(G)$.
        \end{enumerate}
        Independently for each $v \in V(G)$, pick a uniformly random subset $S(v) \subseteq L(v)$ of size $s$. Then, with probability at least $1 - 1/n$, $G$ has a proper $\mathcal{H}$-coloring $\phi$ with $\phi(v) \in S(v)$ for all $v \in V(G)$. 
    \end{corl}
    
    The proof of Corollary~\ref{corl:sparse} is essentially the same as the proof of Theorem~\ref{sparseAA} in \cite[\S3.2]{Palette}, but with Corollary~\ref{mainCorollary} replacing \cite[Proposition 3.2]{Palette}. For completeness, we present it in \S\ref{section:sparse}.

    \subsection{Proof strategy}\label{subsec:sketch}

    Let us now say a few words about the proof of Theorem~\ref{mainTheorem}. Perhaps we should start by noting that, while we perform our arguments in the general setting of DP-coloring, our results are new already for ordinary graph coloring (and, indeed, for independence numbers---see Corollary~\ref{corl:alpha}). That being said, even to just prove an upper bound on the usual chromatic number using our techniques, we must consider the more general framework of list-coloring and work with color-degree constraints. This phenomenon is common in the field and already occurs in the seminal papers of Kim \cite{Kim95} and Johansson \cite{Joh_triangle,Joh_sparse}. Generalizing from list-coloring to DP-coloring does add certain extra complications, which we highlight below, but the bulk of the proof is essentially the same in both settings, so the reader primarily interested in list-coloring may focus on that case and simply ignore the few places where features particular to DP-coloring become relevant.

    Now, let $G$ be a graph and let $\mathcal{H} = (L,H)$ be a DP-cover of $G$ satisfying the assumptions of Theorem~\ref{mainTheorem}; so, $H$ is $K_{1,s,t}$-free and has maximum degree $d$, while $|L(v)| \geq (4+\epsilon)d/\log d \eqqcolon \ell$ for all $v \in V(G)$. In order to find a proper $\mathcal{H}$-coloring of $G$, we employ a variant of the so-called ``R\"odl Nibble'' method, in which we randomly color a small portion of $V(G)$ and then iteratively repeat the same procedure with the vertices that remain uncolored; see \cite{Nibble} for a recent survey of applications of this method to \ep{hyper}graph coloring. In the following paragraphs we highlight some of the key elements of our construction. For the sake of clarity, some of the notation in this overview is slightly different from the one used in the actual proof, but only in minor and technical ways.

    At the heart of our argument is a certain randomized procedure, described in detail in \S\ref{sectionOutline}. This procedure is mostly standard; for example, it is very similar to the one employed in \cite[\S13]{MolloyReed} by Molloy and Reed as well as the one used by Alon and Assadi in \cite[Appendix A]{Palette} \ep{although our procedure is somewhat different since it is applied in the DP-coloring framework}. The output of this procedure is a proper partial $\mathcal{H}$-coloring $\phi$ of $G$. Our hope is to show that, with positive probability, $\phi$ has some desirable properties that allow the coloring to be extended to the entire graph $G$. To explain what these properties are, we need to introduce some notation. Let $G'$ be the subgraph of $G$ induced by the uncolored vertices. Consider an uncolored vertex $v \in V(G')$. Since some of the neighbors of $v$ may already be colored, not every color from $L(v)$ may be available to $v$. Specifically, if $v$ has a neighbor $u \in \dom(\phi)$ with $\phi(u) \sim c \in L(v)$, then $c$ cannot be used for $v$. Thus, we are led to define \[L_\phi(v) \,\defeq\, \set{c \in L(v) \,:\, N_H(c) \cap \im(\phi) = \0}.\] That is, $L_\phi(v)$ is the set of all colors that can be used for $v$ without creating any conflicts with the already colored vertices. For technical reasons, our procedure is ``wasteful,'' in the sense that it may remove more colors from the lists than strictly necessary. Namely, the procedure generates, for each uncolored vertex $v \in V(G')$, a certain subset $L'(v) \subseteq L_\phi(v)$, and we are only allowed to use the colors from $L'(v)$ to color $v$ later on. \ep{See \cite[\S12.2]{MolloyReed} for a textbook discussion of the utility of such ``wastefulness.''}
    Now we let $\mathcal{H}' = (L', H')$ be the DP-cover of $G'$ induced by these restricted lists. Our problem is to show that, with positive probability, the graph $G'$ is $\mathcal{H}'$-colorable.
    
    Let $\ell' \defeq \min \set{|L'(v)| \,:\, v \in V(G')}$ and let $d'$ be the maximum degree of $H'$. Ideally, we would show that, with positive probability, the ratio $d'/\ell'$ is noticeably smaller than $d/\ell$. We would then continue applying our procedure iteratively until the ratio of the maximum degree of the cover graph to the minimum list size drops below a small constant, after which the coloring can be completed using standard tools \ep{such as Proposition~\ref{finalBlow} below}. To achieve this goal, it is natural to start by estimating $\E[|L'(v)|]$ for every vertex $v \in V(G)$ and $\E[\deg_{H'}(c)]$ for every color $c \in V(H)$. As a result of this calculation, it turns out that
    \[
        \frac{\E[\deg_{H'}(c)]}{\E[|L'(v)|]} \,\approx\, \uncolor \, \frac{d}{\ell},
    \]
    where $\uncolor$ is a certain factor strictly less than $1$ \ep{it is defined precisely in \S\ref{sectionOutline}; the exact formula is not important for this informal overview}. The usual next step would be to argue that $|L'(v)| \approx \E[|L'(v)|]$ and $\deg_{H'}(c) \approx \E[\deg_{H'}(c)]$ with high probability. A standard \LLL-based argument would then imply that, with positive probability,
    \[
        \frac{d'}{\ell'} \,\approx\, \uncolor \, \frac{d}{\ell},
    \]
    i.e., $d'/\ell'$ is indeed less than $d/\ell$ by some factor. Unfortunately, this approach fails, because in general $\deg_{H'}(c)$ may not be concentrated around its expected value. Roughly speaking, the problem is that the neighbors of $c$ in $H$ may have many common neighbors.\footnote{This issue does not arise if we assume that $H$ is $K_{s,t}$-free, which is why $K_{s,t}$-free graphs satisfy the version of Theorem~\ref{mainTheorem} with $1$ in place of $4$ in the constant factor \cite{BipartiteNibble}.} As a result, the random events $\set{c' \in V(H')}$ for $c' \in N_H(c)$ may be strongly correlated with each other, which drives up the variance of $\deg_{H'}(c)$. To circumvent this difficulty, we employ an idea introduced by Jamall \cite{Jamall} and developed further by Pettie and Su \cite{PS15} and Alon and Assadi \cite{Palette}. Namely, instead of bounding the \emph{maximum} degree of $H'$ directly, we focus on the \emph{average} value
    \[
        \overline{\deg}_{\mathcal{H}'}(v) \,\defeq\, \frac{1}{|L'(v)|}\sum_{c \in L'(v)}\deg_{H'}(c)
    \]
    for each $v \in V(G')$, called the \emph{average color-degree} of $v$ in $\mathcal{H}'$. (Recall that we briefly mentioned average color-degree at the end of \S\ref{subsec:DP}.) In order to obtain an upper bound on $d'$ in terms of $\overline{\deg}_{\mathcal{H}'}(v)$, we simply remove from $L'(v)$ every color $c$ with $\deg_{H'}(c) > 2\overline{\deg}_{\mathcal{H}'}(v)$. In this way we remove at most half of the colors in $L'(v)$. This is how we gain a factor of $4$ in the lower bound on $\ell$: roughly, a factor of $2$ is added because the degree of a color is allowed to go up to twice the average value, and another factor of $2$ is added because the list sizes can shrink by half. 
    
    Working with the random variable $\overline{\deg}_{\mathcal{H}'}(v)$ instead of dealing with $\deg_{H'}(c)$ for each color $c \in L(v)$ separately leads to a number of technical challenges, which we now address.
    
    \begin{itemize}
        \item From the above discussion, it may seem that we may lose a factor of $4$ in the lower bound on $\ell$ \emph{at every iteration} of the procedure, since every iteration ends with the removal of the high-degree colors. Thankfully, there is a way to get around this issue. The key is to notice that the number of colors removed from $L'(v)$ due to having high degree in $H'$ doesn't have to be equal to $|L'(v)|/2$; rather, it can range anywhere between $0$ and $|L'(v)|/2$. Moreover, the more colors we remove, the smaller the average degree of the remaining colors is. For example, in the extreme case when we actually remove half the colors from $L'(v)$, all the remaining colors must have degree $0$---and thus we don't need to worry about them anymore. To make use of this observation, throughout the iterative process, we maintain the condition that the average color-degree of $v$ is bounded above by a certain increasing function of $|L(v)|$ \ep{i.e., the smaller $|L(v)|$ is, the smaller the average degree of the colors in $L(v)$ becomes}. A careful calculation presented in \S\ref{sectionIterations} shows that the balance between list sizes and average color-degrees results in the loss of only a single factor of $4$ throughout the entire process. Note that, as a result, we cannot guarantee that the list sizes for different vertices will remain equal or even close to each other, which adds certain technical difficulties to the analysis.

        \item\label{pageref for item} The assumption that $H$ is $K_{1,s,t}$-free is used to bound the expectation of $\overline{\deg}_{\mathcal{H}'}(v)$. Roughly speaking, since $H$ is $K_{1,s,t}$-free, the neighborhood of any color $c \in L(v)$ does not induce too many edges in $H$. From this it is possible to deduce that the event $\set{c \in L'(v)}$ and the random variable $\deg_{H'}(c)$ are, in some sense, approximately independent, which helps to estimate the expected contribution of $c$ to the average color-degree of $v$ in $\mathcal{H}'$. The arguments involved here are rather delicate \ep{which is somewhat surprising, since usually computing the expected values is the ``easy'' part of a probabilistic analysis} and constitute a significant portion of the proof (see \S\S\ref{sectionProofOfIterationTheorem} and \ref{sectionProofofExpectation}).
        
        While the argument is already subtle in the list-coloring setting, the DP-coloring framework adds certain further complications.
        Namely, it is possible that distinct colors $c$, $c' \in L(v)$ and $c'' \in N_H(c)$ satisfy $N_H(c') \cap N_H(c'') \neq \0$ (this situation cannot occur in the list-coloring setting).
        This possibility makes the events $\set{|L'(v)| \geq \ell'}$ and $\set{c, c'' \in V(H')}$ more strongly correlated.
        Fortunately, we can use tools such as \hyperref[harris]{Harris's Inequality} to control the influence of this correlation on our parameters (for the details, see Lemma~\ref{keptUncolorExpectation}).
        
        \item 
        \hypertarget{partitioning}{} As mentioned before, most of our arguments remain essentially the same or are only somewhat simplified by working with list-coloring instead of DP-coloring. The one major exception to this is the proof that the random variable $\overline{\deg}_{\mathcal{H}'}(v)$ is highly concentrated around its mean, which is significantly more difficult in the DP-coloring context.
        This is because, in the list-coloring setting, 
        the contributions of the individual colors $c \in L(v)$ to $\overline{\deg}_{\mathcal{H}'}(v)$ are independent from each other, while in the general DP-coloring framework,
        two colors from $L(v)$ may be joined by a path in $H$, creating a dependence between them. Among the tools that we use to tackle this issue and control $\overline{\deg}_{\mathcal{H}'}(v)$ is a \emph{random partitioning technique}. Roughly speaking, in order to prove that some random variable $X$ is concentrated, we randomly represent $X$ as a sum $X = X_1 + \cdots + X_k$ and then show that each piece $X_i$ is highly concentrated around its own expected value. We have employed this technique earlier in \cite{BipartiteNibble}, and expect that it will have further applications in graph coloring. This part of the argument is presented in \S\ref{sectionProofofConcentration}.
    \end{itemize}
    
    \subsubsection*{A road map}
    
    To finish the introduction, let us give a brief outline of the remainder of the paper. In \S\ref{sec:prelim} we collect some preliminary facts and tools. Next, in \S\ref{sectionOutline}, we describe our randomized coloring procedure. In \S\ref{sectionProofOfIterationTheorem} we explain the main steps of the analysis of the procedure. Then, in \S\ref{sectionProofofExpectation} and \S\ref{sectionProofofConcentration}, we study respectively the expected value and the concentration of the average color-degrees of the uncolored vertices. In \S\ref{sectionIterations}, we put the iterative process together and complete the proof of Theorem~\ref{mainTheorem}. Finally, in \S\ref{section:sparse}, we prove Corollary~\ref{corl:sparse}.

\section{Preliminaries}\label{sec:prelim}

    In this section we outline the main background facts that will be used in our arguments. We start with the symmetric version of the Lov\'asz Local Lemma.

\begin{theo}[{Lov\'asz Local Lemma; \cite[\S4]{MolloyReed}}]\label{LLL}
    Let $A_1$, $A_2$, \ldots, $A_n$ be events in a probability space. Suppose there exists $p \in [0, 1)$ such that for all $1 \leq i \leq n$ we have $\P[A_i] \leq p$. Further suppose that each $A_i$ is mutually independent from all but at most $d_{LLL}$ other events $A_j$, $j\neq i$ for some $d_{LLL} \in \N$. If $4pd_{LLL} \leq 1$, then with positive probability none of the events $A_1$, \ldots, $A_n$ occur.
\end{theo}

Aside from the Local Lemma, we will require several concentration of measure bounds. The first of these is the Chernoff Bound for binomial random variables. We state the two-tailed version below:

\begin{theo}[{Chernoff Bound; \cite[\S5]{MolloyReed}}]\label{chernoff}
    Let $X$ be the binomial random variable on $n$ trials with each trial having probability $p$ of success. Then for any $0 \leq \xi \leq \E[X]$, we have
    \begin{align*}
        \P\Big[\big|X - \E[X]\big| \geq \xi\Big] < 2\exp{\left(-\frac{\xi^2}{3\E[X]}\right)}. 
    \end{align*}
\end{theo}

    Additionally, we will take advantage of a version of Talagrand's inequality developed by Bruhn and Joos \cite{ExceptionalTal}. We refer to it as \emph{Exceptional Talagrand's Inequality}, since its main feature is the use of a set $\Omega^\ast$ of ``exceptional outcomes.'' 

\begin{theo}[{Exceptional Talagrand's Inequality \cite[Theorem 12]{ExceptionalTal}}]\label{ExceptionalTalagrand}
    Let $X$ be a non-negative random variable, not identically 0, which is a function of $n$ independent trials $T_1$, \ldots, $T_n$, and let $\Omega$ be the set of outcomes for these trials. Let $\Omega^* \subseteq \Omega$ be a measurable subset, which we shall refer to as the \emphd{exceptional set}. Suppose that $X$ satisfies the following for some $\gamma>1$, $s>0$:
    \begin{enumerate}[label=\ep{\normalfont{}ET\arabic*}]
        \item\label{item:ET1} For all $q>0$ and every outcome $\omega \notin \Omega^*$, there is a set $I$ of at most $s$ trials such that $X(\omega') > X(\omega) - q$ whenever $\omega' \not\in \Omega^*$ differs from $\omega$ on fewer than $q/\gamma$ of the trials in $I$.
        \item\label{item:ET2} $\P[\Omega^*] \leq M^{-2}$, where $M = \max\{\sup X, 1\}$.
    \end{enumerate}
    Then for every $\xi > 50\gamma\sqrt{s}$, we have:
    \begin{align*}
        \P\Big[\big|X - \E[X]\big| \geq \xi\Big] \leq 4\exp{\left(-\frac{\xi^2}{16\gamma^2s}\right)} + 4\P[\Omega^*].
    \end{align*}
\end{theo}

    We will also need the following special case of the FKG inequality, dating back to Harris \cite{Harris} and Kleitman \cite{Kleitman}:

\begin{theo}[{Harris's inequality/Kleitman's Lemma \cite[Theorem 6.3.2]{AlonSpencer}}]\label{harris}
    Let $X$ be a finite set and let $S \subseteq X$ be a random subset of $X$ obtained by selecting each $x \in X$ independently with probability $p_x \in [0,1]$. If $\mathcal{A}$ \ep{resp.~$\mathcal{B}$} is an increasing \ep{resp.~decreasing} family of subsets of $X$, then \[\P[S \in \mathcal{A} \,\vert\, S \in \mathcal{B}] \,\leq\, \P[S \in \mathcal{A}].\]
\end{theo}

Finally, we shall use the K\H{o}v\'ari--Sós--Turán theorem for $K_{s,t}$-free graphs:

\begin{theo}[K\H{o}v\'ari--Sós--Turán \cite{KST}; see also \cite{KST2}]\label{KST}
    Let $G$ be a bipartite graph with a bipartition $V(G) = X \sqcup Y$, where $|X| = m$, $|Y| = n$, and $m \geq n$. Suppose that $G$ does not contain a complete bipartite subgraph with $s$ vertices in $X$ and $t$ vertices in $Y$.
    Then
    $|E(G)| \leq s^{1/t} m^{1-1/t} n + tm$.
\end{theo}

\section{The Coloring Procedure}\label{sectionOutline}

    Let $G$ be a graph with a DP-cover $\mathcal{H} = (L, H)$. Most of the proof of Theorem \ref{mainTheorem} is concerned with showing that, under suitable assumptions, $G$ has a proper partial $\mathcal{H}$-coloring with certain desirable properties. We start by introducing some notation.
    For a vertex $v\in V(G)$, the \emphd{average color-degree} $\overline{\deg}_\mathcal{H}(v)$ of $v$ in $\mathcal{H}$ is defined by the formula 
    \[
        \overline{\deg}_{\mathcal{H}}(v) \,\defeq\, \frac{1}{|L(v)|}\sum_{c \in L(v)}\deg_H(c).
    \]
    Given a partial $\mathcal{H}$-coloring $\phi$ of $G$ and a vertex $v \in V(G)$, we let
    \[
        L_\phi(v) \,\defeq\, \{c \in L(v) \,:\, N_H(c) \cap \im(\phi) = \0 \}.
    \]
    As explained in \S\ref{subsec:sketch}, if $v \not \in \dom(\phi)$, then $L_\phi(v)$ is the set of colors that can be used for $v$ without creating a conflict with the vertices that are already colored by $\phi$. The main tool in our proof of Theorem~\ref{mainTheorem} is the following technical lemma:

\begin{Lemma}\label{iterationTheorem} 
    There are $\tilde{d} \in \N$, $\tilde{\alpha} > 0$ such that the following holds. Let $\eta$, $d$, $\ell > 0$, $s$, $t\in \N$ satisfy:
    \begin{enumerate}[label=\ep{\normalfont\arabic*}]
        \item $d$ is sufficiently large: $d \geq \tilde{d}$,
        \item\label{item:ell} $\ell$ is bounded below and above in terms of $d$: $4\eta\, d < \ell < 100d$,
        \item\label{item:t} $s$ and $t$ are bounded in terms of $d$: $s \leq d^{1/10}$ and 
        $t  \leq \dfrac{\tilde{\alpha}\log d}{\log \log d}$, 
        \item\label{item:eta} $\eta$ is not far below $1/\log d$: $\dfrac{1}{\log^5d} < \eta < \dfrac{1}{\log d}.$
    \end{enumerate}
    Let $G$ be a graph with a DP-cover $\mathcal{H} = (L, H)$ such that for some $\beta$ satisfying $d^{-1/(200t)} \leq \beta \leq 1/10$,
    \begin{enumerate}[label=\ep{\normalfont\arabic*},resume]
        \item $H$ is $K_{1,s,t}$-free,
        \item\label{item:Delta} $\Delta(H) \leq 2d$,
        \item\label{item:list_assumption} the list sizes are roughly between $\ell/2$ and $\ell$: $(1-\beta)\ell/2 \,\leq\, |L(v)| \leq (1+\beta)\ell$ for all $v \in V(G)$,
        \item\label{item:averaged} average color-degrees are smaller for vertices with smaller lists of colors: \[\overline{\deg}_\mathcal{H}(v) \,\leq\, \left(2 - (1 - \beta)\frac{\ell}{|L(v)|}\right)d \quad \text{for all } v\in V(G).\]
    \end{enumerate}
    Then there exist a proper partial $\mathcal{H}$-coloring $\phi$ of $G$ and 
    an assignment of subsets $L'(v) \subseteq L_\phi(v)$ to each $v \in V(G) \setminus \dom(\phi)$ with the following properties. Let
    \[
        G' \defeq G\left[V(G)\setminus \dom(\phi)\right] \qquad \text{and} \qquad H' \defeq H\left[\textstyle\bigcup_{v \in V(G')} L'(v)\right].
    \]
    Define the following quantities:
    \[
        \begin{aligned}[c]
            \keep &\defeq \left(1 - \frac{\eta}{\ell}\right)^{2d}, \\
        \uncolor &\defeq \left(1  - \frac{\eta}{\ell}\right)^{\keep\,\ell/2},
        \end{aligned}
        \qquad
        \begin{aligned}[c]
            \ell' &\defeq \keep\, \ell, \\
            d' &\defeq \keep\, \uncolor\, d,
        \end{aligned} \qquad \beta' \defeq (1+36\eta)\beta.
    \]
    Let $\mathcal{H}' \defeq (L', H')$, so $\mathcal{H}'$ is a DP-cover of $G'$. 
    Then for all $v \in V(G')$: 
    \begin{enumerate}[label=\ep{\normalfont\roman*}]
        \item\label{item:I} $|L'(v)| \,\leq\, (1+\beta')\ell'$,
        
        \smallskip
        
        \item\label{item:II} $|L'(v)| \geq (1-\beta')\ell'/2$,
        
        \smallskip
        
        \item\label{item:III} $\Delta(H') \leq 2d'$, 
        
        \smallskip
        
        \item\label{item:IV} $\overline{\deg}_{\mathcal{H}'}(v) \leq \left(2 - (1 - \beta')\frac{\ell'}{|L'(v)|}\right)d'.$
    \end{enumerate}
\end{Lemma}

    We remark that above and in what follows, we let $\keep\,\uncolor$ indicate the product of the terms $\keep$ and $\uncolor$.
    Notice that conditions \ref{item:I}--\ref{item:IV} in Lemma~\ref{iterationTheorem} are analogous to the assumptions \ref{item:Delta}--\ref{item:averaged}, but with $\ell$, $d$, and $\beta$ replaced by $\ell'$, $d'$, and $\beta'$. This will allow us to apply Lemma~\ref{iterationTheorem} iteratively in \S\ref{sectionIterations}. To prove Lemma \ref{iterationTheorem}, we employ a variant of Alon and Assadi's ``Wasteful Coloring Procedure'' used in \cite[Appendix A]{Palette}, where it is applied to triangle-free graphs. \ep{The parameters $\keep$ and $d'$ in this procedure are defined as in Lemma~\ref{iterationTheorem}.}\\ 

    \begin{mdframed}
    \hypertarget{procedure}{\textit{The Coloring Procedure}}

    \medskip
    
    \begin{mdframed}
    \noindent \textbf{Input:} A graph $G$ with a DP-cover $\mathcal{H} = (L,H)$ and parameters $\eta \in [0,1]$ and $d$, $\ell> 0$.

    \smallskip

    \noindent \textbf{Output:}
    A proper partial $\mathcal{H}$-coloring $\phi$ and subsets $L'(v) \subseteq L_\phi(v)$ for all $v \in V(G)$.
    \end{mdframed}
    
    \smallskip

    \begin{enumerate}[label=\ep{\normalfont{}S\arabic*}]
    \item For every $v \in V(G)$, generate a random subset $A(v) \subseteq L(v)$ by putting each color $c \in L(v)$ into $A(v)$ independently with probability $\eta/\ell$. The colors in $A(v)$ are said to be \emphdef{activated}, and we let $A \defeq \bigcup_{v \in V(G)} A(v)$ be the set of all activated colors.
    
    \smallskip
    
    \item\label{step:eq} Let $(\eq(c) \,:\, c \in V(H))$ be a family of independent random variables 
    with distribution
    \[\eq(c) \sim \Ber\left(\keep  /\left(1 - \eta/\ell\right)^{\deg_H(c)}\right).\]
    We call these variables the \emphd{equalizing coin flips}.
    
    \smallskip
    
    \item For each vertex $v \in V(G)$, let
    \[
        K(v) \,\defeq\, \set{c \in L(v) \,:\, N_H(c) \cap A = \0,\, \eq(c) = 1}.
    \]
    Let $K \defeq \bigcup_{v \in V(G)}K(v)$. We call the colors in $K$ \emphd{kept} and the ones not in $K$ \emphd{removed}.  

    \smallskip
    
    \item For each $v \in V(G)$, if $A(v) \cap K(v) \neq \0$, then we set $\phi(v)$ to any color in $A(v) \cap K(v)$. For all other vertices, $\phi(v)$ is undefined. For convenience, we write $\phi(v) = \mathsf{blank}$ if $\phi(v)$ is undefined, so $\dom(\phi) = \set{v \,:\, \phi(v) \neq \mathsf{blank}}$. We call the vertices $v \not \in \dom(\phi)$ \emphdef{uncolored} and define 
    \[U \,\defeq\, \set{c \in V(H)\,:\, \phi\left(L^{-1}(c)\right) = \blank}.\]
    \ep{Recall that $L^{-1}(c)$ denotes the underlying vertex of $c$ in $G$.}
    
    \smallskip
    
    \item\label{item:step5} For each vertex $v \in V(G)$, we let
    \[
    L'(v) \,\defeq\, \set{c\in K(v)\,:\, |N_H(c) \cap K \cap U| \leq 2\,d'}.
    \]
\end{enumerate}
\end{mdframed}
\hspace{2pt}

    Notice that, under assumption \ref{item:Delta} of Lemma~\ref{iterationTheorem},
    \[
        0 \,\leq\, \frac{\keep}{\left(1 - \eta/\ell\right)^{\deg_H(c)}} \,=\, \left(1 - \frac{\eta}{\ell}\right)^{2d - \deg_H(c)} \,\leq\, 1
    \]
    for every color $c \in V(H)$, so the equalizing coin flips in step \ref{step:eq} are well-defined.

    In \S\S\ref{sectionProofOfIterationTheorem}, \ref{sectionProofofExpectation}, and \ref{sectionProofofConcentration}, we will show that, with positive probability, the output of the \hyperlink{procedure}{Coloring Procedure} satisfies the conclusion of Lemma~\ref{iterationTheorem}. \ep{Notice that, for convenience, the \hyperlink{procedure}{Coloring Procedure} defines $L'(v)$ for every vertex $v \in V(G)$, although we only use this notation for uncolored vertices in the statement of Lemma~\ref{iterationTheorem}.} With this procedure in mind, we can now explain the meaning of the quantities $\keep$ and $\uncolor$. 
Suppose $G$ and $\mathcal{H} = (L,H)$ satisfy the conditions of Lemma~\ref{iterationTheorem}. If we run the \hyperlink{procedure}{Coloring Procedure} with these $G$ and $\mathcal{H}$, then $\keep$ is precisely the probability that a color $c \in V(H)$ is kept \ep{i.e., $c \in K$}, while $\uncolor$ is an upper bound on the probability that a vertex $v \in  V(G)$ is uncolored \ep{we shall verify these facts in \S\S\ref{sectionProofOfIterationTheorem} and~\ref{sectionProofofExpectation}}.
As mentioned in \S\ref{subsec:sketch}, we can write
\[
    \frac{d'}{\ell'} \,=\, \uncolor\, \frac{d}{\ell},
\]
which intuitively means that the ratio between the degrees of the colors and the sizes of the lists ``ought to'' decrease by a factor of $\uncolor$. Unfortunately, as explained in \S\ref{subsec:sketch}, we cannot guarantee that it actually does decrease this much, which forces us to instead keep track of a more complicated relationship between list sizes and average color-degrees. Nevertheless, we will show in \S\ref{sectionIterations} that after applying Lemma~\ref{iterationTheorem} iteratively $O(\log d \log \log d)$ times, we will be able to complete the coloring using the
following proposition:

\begin{prop}\label{finalBlow}
    Let $G$ be a graph with a $DP$-cover $\mathcal{H} = (L,H)$ such that $|L(v)| \geq 8d$ for every $v \in V(G)$, where $d$ is the maximum degree of $H$. Then there exists a proper $\mathcal{H}$-coloring of $G$.
\end{prop}

This proposition is standard and proved using the \hyperref[LLL]{Lov\'asz Local Lemma}. Its first appearance \ep{in the list-coloring setting} is in the paper \cite{Reed} by Reed; see also \cite[\S4.1]{MolloyReed} for a textbook treatment. (A proof in the DP-coloring framework can be found in \cite[Appendix]{JMTheorem}.)
Note that a result of Haxell allows one to replace the constant $8$ with $2$ in the list coloring setting \cite{haxell2001note}.

\section{Proof of Lemma \ref{iterationTheorem}}\label{sectionProofOfIterationTheorem}

    Suppose parameters $\eta$, $d$, $\ell$, $s$, $t$, $\beta$, and a graph $G$ with a DP-cover $\mathcal{H} = (L, H)$ satisfy the assumptions of Lemma~\ref{iterationTheorem}. Throughout, we shall assume that $\tilde{d}$ is sufficiently large, while $\tilde{\alpha}$ is sufficiently small. Since we may remove all the edges of $G$ whose corresponding matchings in $H$ are empty, we may assume that $\Delta(G) \leq 2(1+\beta) \ell d$. Let the quantities $\keep$, $\uncolor$, $d'$, $\ell'$, and $\beta'$ be defined as in the statement of Lemma~\ref{iterationTheorem}. Suppose we have carried out the \hyperlink{procedure}{Coloring Procedure} with these $G$ and $\mathcal{H}$. As in the statement of Lemma~\ref{iterationTheorem}, we let
    \[
        G' \defeq G\left[V(G)\setminus \dom(\phi)\right],  \qquad H' \defeq H\left[\textstyle\bigcup_{v \in V(G')} L'(v)\right], \qquad \text{and} \qquad \mathcal{H}' \defeq (L', H').
    \]
    For each $v \in V(G)$, we define the following quantities:
    \[
        \ell(v) \defeq |L(v)| \qquad \text{and} \qquad \overline{\deg}(v) \defeq \overline{\deg}_\mathcal{H}(v),
    \]
    as well as the following random variables:
    \[
        k(v) \defeq |K(v)|, \qquad  \ell'(v) \defeq |L'(v)|, \qquad \text{and} \qquad \overline{d}(v) \defeq \frac{1}{\ell'(v)}\sum_{c\in L'(v)}|N_H(c) \cap V(H')|.
    \]
    By definition, if $v \in V(G')$, then $\overline{d}(v) = \overline{\deg}_{\mathcal{H}'}(v)$. Our goal is to verify that statements \ref{item:I}--\ref{item:IV} in Lemma \ref{iterationTheorem} hold for every $v \in V(G')$ with positive probability. To assist with this task we follow an idea of Pettie and Su \cite{PS15} and define the following auxiliary quantities:
    \begin{align*}
        \lambda(v) &\defeq \frac{\ell(v)}{\ell}, & \lambda'(v)&\defeq \frac{\ell'(v)}{\ell'}, \\
        \delta(v) &\defeq \lambda(v)\,\overline{\deg}(v) + (1-\lambda(v))\,2d, & \delta'(v) &\defeq \lambda'(v)\,\overline{d}(v) + (1-\lambda'(v))\,2d'.
    \end{align*}
    Note that, by \ref{item:list_assumption}, we have $(1-\beta)/2 \leq \lambda(v) \leq 1+ \beta$. When $\lambda(v) \leq 1$, we can think of $\delta(v)$ as what the average color-degree of $v$ would become if we add $\ell - \ell(v)$ colors of degree $2d$ to $L(v)$. We remark that in both \cite{PS15} and \cite{Palette}, the value $\lambda(v)$ is artificially capped at $1$. However, it turns out that there is no harm in allowing $\lambda(v)$ to exceed $1$, which moreover makes the analysis somewhat simpler. The upper bound on $\overline{\deg}(v)$ given by \ref{item:averaged} implies that
    \begin{align}\label{eqn:avg_delta}
        \delta(v) \,\leq\, (1+\beta)d.
    \end{align}
    It turns out that an upper bound on $\delta'(v)$ suffices to derive statements \ref{item:II}--\ref{item:IV} in Lemma \ref{iterationTheorem}:

\begin{Lemma}\label{lemma:deltaprime}
    If $\delta'(v) \leq (1 + \beta')d'$, then conditions \ref{item:II}--\ref{item:IV} of Lemma \ref{iterationTheorem} are satisfied.
\end{Lemma}

\begin{proof}
    From Step \ref{item:step5} of the \hyperlink{procedure}{Coloring Procedure}, we have $\Delta(H') \leq 2d'$ and so condition \ref{item:III} always holds. Now take any $v \in V(G')$. By the definition of $\delta'(v)$, we have
\[\delta'(v) \geq (1-\lambda'(v))2d' \quad \implies \quad \lambda'(v) \geq 1 - \frac{\delta'(v)}{2d'} \geq \frac{1 - \beta'}{2}.\]
From here, we can bound $\ell'(v)$ as follows:
\[\ell'(v) \,=\, \lambda'(v)\,\ell' \,\geq\, (1 - \beta')\frac{\ell'}{2}.\]
    Thus, condition \ref{item:II} holds. From the inequality
    \[
        \lambda'(v)\,\overline{d}(v) + (1-\lambda'(v))\,2d' \,=\, \delta'(v) \,\leq\, (1 + \beta')d',
    \]
    it follows that
    \[
        \overline{d}(v) \,\leq\, \frac{1}{\lambda'(v)} \, (2\lambda'(v) -(1 - \beta'))d'  \,=\, \left(2 -(1 - \beta')\frac{\ell'}{\ell'(v)}\right)d',
    \]
    and hence condition \ref{item:IV} holds as well.
\end{proof}

Therefore, to prove Lemma \ref{iterationTheorem}, it suffices to show that, with positive probability, the outcome of the \hyperlink{procedure}{Coloring Procedure} satisfies $\delta'(v) \leq (1 + \beta')d'$ and $\ell'(v) \leq (1+\beta')\ell'$ for all $v \in V(G)$. We shall now prove some intermediate results. Before we do so, consider the following inequality, which will be useful for proving certain bounds and follows for small enough $\tilde\alpha$:
\begin{align}\label{etabeta}
    \eta\beta \,\geq\, d^{-1/(200t)}/\log^5d \,\geq\, d^{-1/(100t)}.
\end{align}

In the next lemma, we show that $k(v)$ is concentrated around its expected value, which in particular implies that condition \ref{item:I} in Lemma \ref{iterationTheorem} is satisfied with high probability. 

\begin{Lemma}\label{listConcentration}
    $\P[|k(v) -\keep\,\ell(v)| \geq \eta\beta\,\keep\,\ell(v)] \leq \exp\left(-d^{1/10}\right)$.
\end{Lemma}
\begin{proof}
    Recall that $K(v)$ is the set of colors $c \in L(v)$ that have no activated neighbors and satisfy $\eq(c) = 1$. By the definition of the equalizing coin flip $\eq(c)$, the probability of this happening for a given color $c$ is precisely $\keep$. Therefore, $\E[k(v)] = \keep\,\ell(v)$. Recall that edges in $G$ correspond to matchings in $H$ and so, no two colors in $L(v)$ share a neighbor. Furthermore, all activations and equalizing coin flips happen independently and therefore we may apply the \hyperref[chernoff]{Chernoff Bound} \ep{Theorem \ref{chernoff}} to get for $d$ large enough:
    \begin{align*}
        \P[|k(v)-\keep\,\ell(v)| \geq \eta\beta\,\keep\,\ell(v)] &\leq 2\exp\left(-\frac{\eta^2\beta^2\,\keep\,\ell(v)}{3}\right)\\
        &\leq 2\exp\left(-\frac{d^{1 - 1/(25t)}}{6}\right) \\
        &\leq \exp\left(-d^{1/10}\right),
    \end{align*}
    where the second inequality follows since  $\keep \geq 1 - \frac{2d\eta}{\ell} \geq \frac{1}{2}$ (here we are using the inequality $(1-x)^y \geq 1-yx$ for $y\geq 1$, and  \ref{item:ell}), by inequality \eqref{etabeta}, and because
    \[
        \ell(v) \,\geq\, (1-\beta)\ell/2 \,\geq\, 2(1-\beta) \eta\, d \,\geq\,  2(1-\beta)\,\frac{d}{\log^5d} \,\geq\, d^{1-1/(50t)}
    \]
    by \ref{item:list_assumption}, \ref{item:ell}, \ref{item:eta}, and \ref{item:t} \ep{assuming $\tilde{\alpha}$ is small enough}.
\end{proof}

    Since $\ell'(v)\leq k(v)$, we have the following with probability at least $1-\exp\left(-d^{1/10}\right)$:
    \begin{align*}
        \ell'(v) &\leq (1+\eta\beta)\,\keep\,\ell(v) \\
        &\leq (1+\eta\beta)\,(1+\beta)\,\keep\,\ell \\
        &\leq (1+ \beta')\,\ell'.
    \end{align*}
    This implies that condition \ref{item:I} is met. 
    
In order to analyze the average color-degrees in the DP-cover $\mathcal{H}'$, we define:
\begin{align*}
    \keptedges &\defeq \{cc'\in E(H)\,:\, c\in L(v),\ c'\in N_H(c), \text{ and } c,c'\in K\}, \\
    \uncoloredges &\defeq \{cc'\in E(H)\,:\, c\in L(v),\ c'\in N_H(c), \text{ and }
    c' \in U\}, \\
    \nd &\defeq \frac{|\keptedges \cap \uncoloredges|}{k(v)}.
\end{align*}
Note that $\nd$ is what the average color-degree of $v$ would be if instead of removing colors with too many neighbors on Step \ref{item:step5} of the \hyperlink{procedure}{Coloring Procedure}, we had just set $L'(v) = K(v)$.
    The only place in our proof that relies on $H$ being $K_{1,s,t}$-free is the following lemma, which gives a bound on the expected value of $|\keptedges \cap \uncoloredges|$.
    
\begin{Lemma}\label{expectationKeptUncolor}
$\E[|\keptedges\cap\uncoloredges|] \leq \keep^2\,\uncolor\,\ell(v)\,\overline{\deg}(v)(1+6\eta\beta)$.
\end{Lemma}

\begin{proof}
See \S \ref{sectionProofofExpectation}.
\end{proof}

    Let $d_{\max} \defeq \max\{d^{7/8}, \Delta(H)\}$. In the following lemma, we show that with high probability the quantity $|\keptedges\cap\uncoloredges|$ is not too large.

\begin{Lemma}\label{concentrationKeptUncolor}
$\P[|\keptedges\cap\uncoloredges| \geq \keep^2\,\uncolor\,\ell(v)(\overline{\deg}(v) + 8\eta\beta\,d_{\max})] \leq d^{-100}$.
\end{Lemma}

\begin{proof}
See \S \ref{sectionProofofConcentration}.
\end{proof}

    Using Lemmas \ref{expectationKeptUncolor} and \ref{concentrationKeptUncolor}, we can now prove the following:
\begin{Lemma}\label{degreeConcentration}
    $\P[\nd > \keep\,\uncolor\,\overline{\deg}(v) + 15\eta\beta\,\keep\,\uncolor\,d_{\max}] \leq d^{-75}$.
\end{Lemma}
\begin{proof}
    With probability at least
    \[
        1 - \exp\left(-d^{1/10}\right) - d^{-100} \,\geq\, 1 - d^{-75},
    \]
    the events described in Lemmas \ref{listConcentration} and \ref{concentrationKeptUncolor} do not happen. Thus, with probability at least $1 - d^{-75}$,
\begin{align*}
    \nd &= \frac{|\keptedges\cap \uncoloredges|}{k(v)} \\
    &\leq \frac{\keep^2\,\uncolor\,\ell(v)\,\overline{\deg}(v) + 8\eta\beta\,\keep^2\,\uncolor\,\ell(v)d_{\max}}{(1-\eta\beta)\keep\,\ell(v)} \\
    &\leq \frac{\keep\,\uncolor\,\overline{\deg}(v)}{(1-\eta\beta)} + \frac{8\eta\beta\,\keep\,\uncolor\, d_{\max}}{(1-\eta\beta)} \\
    &\leq \keep\,\uncolor\, \overline{\deg}(v)(1+ 2\eta\beta) + 8\eta\beta(1+2\eta\beta)\keep\,\uncolor \,d_{\max} \\
    &\leq \keep\,\uncolor\, \overline{\deg}(v) + 15\eta\beta\,\keep\,\uncolor \,d_{\max}. \qedhere
\end{align*}
\end{proof}

We can now combine Lemmas \ref{listConcentration} and \ref{degreeConcentration} to prove the desired bound on $\delta'(v)$.
\begin{Lemma}\label{deltaBound}$\P\left[\delta'(v) \leq (1+\beta')d'\right] \geq 1-d^{-50}$.
\end{Lemma}

\begin{proof}\stepcounter{ForClaims} \renewcommand{\theForClaims}{\ref{deltaBound}}

Let us define some new parameters similar to $\lambda$ and $\delta$:
\begin{align*}
    \hat{\lambda}(v) \defeq \frac{k(v)}{\ell'} \qquad \text{and} \qquad  \hat{\delta}(v) \defeq \hat{\lambda}(v)\,\nd + (1-\hat{\lambda}(v))2d'.
\end{align*}
    Note that
    \begin{align*}
        \hat{\delta}(v) - \delta'(v) \,&=\, \frac{1}{\ell'} \, \left(k(v) \nd - \ell'(v) \overline{d}(v) - 2d'(k(v) - \ell'(v))\right) \\
        &=\, \frac{1}{\ell'} \, \left(|\keptedges \cap \uncoloredges| - \sum_{c\in L'(v)}|N_H(c) \cap V(H')| - 2d'(k(v) - \ell'(v))\right) \\
        &\geq\, \frac{1}{\ell'} \, \left(|\keptedges \cap \uncoloredges| - \sum_{c\in L'(v)}|N_H(c) \cap K \cap U| - 2d'(k(v) - \ell'(v))\right).
    \end{align*}
    Since every color $c$ that is in $K(v)$ but not in $L'(v)$ was removed because $|N_H(c) \cap K \cap U| > 2d'$,
    \[
        |\keptedges \cap \uncoloredges| - \sum_{c\in L'(v)}|N_H(c) \cap K \cap U| \,\geq\, 2d'(k(v) - \ell'(v)).
    \]
    It follows that $\delta'(v) \leq \hat{\delta}(v)$. Now we proceed to bound $\hat{\delta}(v)$:

\begin{claim}\label{claim:deltahat}
$\P\left[\hat{\delta}(v) \leq \lambda(v)\,\nd + (1-\lambda(v))2d' + 3\eta\beta\,d'\right] \geq 1 - d^{-60}.$
\end{claim}
\begin{claimproof} Let us assume that the events described in Lemmas \ref{listConcentration} and \ref{degreeConcentration} do not occur, which happens with probability at least
\[
    1 - \exp\left(-d^{1/10}\right) - d^{-75} \,=\, 1 - d^{-60}.
\]
From Lemma~\ref{listConcentration}, we get
\[\hat{\lambda}(v) \,=\, \frac{k(v)}{\ell'} \,\geq\, \frac{(1-\eta\beta)\keep\,\ell(v)}{\keep\,\ell} \,\geq\, (1-\eta\beta)\lambda(v).\]
By definition of $\lambda(v)$ and since $\overline{\deg}(v) \leq \Delta(H) \leq 2d$, we have
\begin{align*}
    \delta(v) &= \lambda(v)\left(\overline{\deg}(v) - 2d\right) + 2d \\
    &\geq (1+\beta)\left(\overline{\deg}(v) - 2d\right) + 2d \\
    &= (1+\beta)\overline{\deg}(v) - 2\beta\,d.
\end{align*}
It follows from \eqref{eqn:avg_delta} that
\[\overline{\deg}(v) \leq \frac{\delta(v)}{1+\beta} + \frac{2\beta \,d}{1+\beta} \leq (1 + 2\beta)d.\]
From the above inequality along with Lemma \ref{degreeConcentration}, we obtain
\begin{align*}
    \nd &\leq \keep\,\uncolor\,\overline{\deg}(v) + 15\eta\beta\,\keep\,\uncolor\,d_{\max} \\
    &\leq (1+2\beta) \, \keep\,\uncolor\,d + 30\eta\beta\,\keep\,\uncolor\,d \\
    &\leq (1+2\beta(1+15\eta))d' \\
    &\leq 2d'.
\end{align*}
The last inequality follows since $\beta \leq 1/10$. Finally, we have: 
\begin{align*}
    \hat{\delta}(v) &= \hat{\lambda}(v)\,\nd + (1-\hat{\lambda}(v))2d' \\
    &\leq (1-\eta\beta)\lambda(v)\,\nd + (1-(1-\eta\beta)\lambda(v))2d' \\
    &= \lambda(v)\,\nd + (1-\lambda(v))2d' + \eta\beta\lambda(v)(2d' - \nd) \\
    &\leq \lambda(v)\,\nd + (1-\lambda(v))2d' + 2(1+\beta)\eta\beta\, d' \\
    &\leq \lambda(v)\,\nd + (1-\lambda(v))2d' + 3\eta\beta\, d',
\end{align*}
    as desired.
\end{claimproof}

From Lemmas \ref{listConcentration} and \ref{degreeConcentration} and Claim~\ref{claim:deltahat}, we conclude that with probability at least
\[
    1 - \exp\left(-d^{1/10}\right) - d^{-75} - d^{-60} \,\geq\, 1 - d^{-50},
\]
the following holds:
\begin{align*}
    \delta'(v) &\leq \hat{\delta}(v) \\
    &\leq \lambda(v)\,\nd + (1-\lambda(v))2d' + 3\eta\beta\, d' \\
    &\leq \lambda(v)\left(\keep\,\uncolor\,\overline{\deg}(v) + 15\eta\beta\,\keep\,\uncolor\,d_{\max}\right) + (1-\lambda(v))2d' + 3\eta\beta\, d' \\
    &\leq \lambda(v)\left(\keep\,\uncolor\,\overline{\deg}(v) + 30\eta\beta\,d'\right) + \left(1-\lambda(v)\right)2d' + 3\eta\beta\, d' \\
    &= \lambda(v)\keep\,\uncolor\,\overline{\deg}(v) + 30\eta\beta\,\lambda(v)\,d' + \left(1-\lambda(v)\right)2\,\keep\,\uncolor\, d + 3\eta\beta\, d' \\    
    &\leq \keep\,\uncolor \left(\lambda(v)\,\overline{\deg}(v) + (1-\lambda(v))2d\right) + 36\eta\beta\,d' \\
    &= \keep\,\uncolor\,\delta(v) + 36\eta\beta\,d' \\
    &\leq (1 + \beta')d'.
\end{align*}
    This completes the proof of Lemma \ref{deltaBound}.
\end{proof}

We are now ready to finish the proof of Lemma \ref{iterationTheorem} \ep{modulo Lemmas~\ref{expectationKeptUncolor} and \ref{concentrationKeptUncolor}}.

\begin{proof}[Proof of Lemma \ref{iterationTheorem}]
    We perform the \hyperlink{procedure}{Coloring Procedure} and define the following random events for each $v \in V(G)$: 
\begin{enumerate}
    \item $A_v \defeq \set{\ell'(v) \leq (1+\beta')}$,
    \item $B_v \defeq \set{\delta'(v) \geq (1+\beta')d'}$.
\end{enumerate}
We will use the \hyperref[LLL]{\LLL} \ep{Theorem \ref{LLL}}. By Lemmas \ref{listConcentration} and \ref{deltaBound}, we have:
\[\P[A_v] \leq \exp\left(-d^{1/10}\right) \leq d^{-50}, \qquad
\P[B_v] \leq d^{-50}.\]
Let $p \defeq d^{-50}$.
    Note that the events $A_v$, $B_v$ are mutually independent from the events of the form $A_u$, $B_u$, where $u \in V(G)$ is at distance more than $4$ from $v$. Since $\Delta(G) \leq 2(1+\beta)\ell d$, there are at most $2(2(1+\beta)\ell d)^4 \leq d^{10}$ events corresponding to the vertices at distance at most $4$ from $v$. So we let $d_{LLL} \defeq d^{10}$ and observe that $4pd_{LLL} = 4 d^{-40} < 1$. By the \hyperref[LLL]{\LLL}, with positive probability none of the events $A_v$, $B_v$ occur. By Lemma~\ref{lemma:deltaprime}, this implies that, with positive probability, the output of the \hyperlink{procedure}{Coloring Procedure} satisfies the conclusion of Lemma~\ref{iterationTheorem}. 
\end{proof}

\section{Proof of Lemma \ref{expectationKeptUncolor}}\label{sectionProofofExpectation}

In this section, we prove Lemma \ref{expectationKeptUncolor}, which we restate below for convenience (see \S \ref{sectionProofOfIterationTheorem} for the definitions of $E_K(u)$ and $E_K(v)$):

\begin{Lemma*}[\ref{expectationKeptUncolor}]
$\E[|\keptedges\cap\uncoloredges|] \leq \keep^2\,\uncolor\,\ell(v)\,\overline{\deg}(v)(1+6\eta\beta)$.
\end{Lemma*}

The following lemma provides an upper bound on $\E[|\keptedges|]$ and is the only part of the proof that requires $K_{1,s,t}$-freeness. Recall that
\[
    \keptedges = \{cc'\in E(H)\,:\, c\in L(v),\ c'\in N_H(c), \text{ and } c,c'\in K\}.
\]

\begin{Lemma}\label{keptExpectation}

$\E[|\keptedges|] \leq \keep^2\,\ell(v)\,\overline{\deg}(v)(1+\eta\beta)$.

\end{Lemma}

\begin{proof}\stepcounter{ForClaims} \renewcommand{\theForClaims}{\ref{keptExpectation}}
    Observe the following:
    \[|\keptedges| = \sum\limits_{c\in L(v)}\sum\limits_{c'\in N_H(c)}\mathbbm{1}_{\{c,c'\in K\}}.\]
    Since all activations and equalizing coin flips occur independently, we have:
\begin{align*}
    \E[|\keptedges|] &= \sum\limits_{c\in L(v)}\sum\limits_{c'\in N_H(c)}\P[c,c'\in K] \\
    &= \sum\limits_{c\in L(v)}\sum\limits_{c'\in N_H(c)}\P[\eq(c) = 1]\,\P[\eq(c') = 1]\left(1 - \frac{\eta}{\ell}\right)^{|N_H(c)\cup N_H(c')|} \\
    &= \keep^2\sum\limits_{c\in L(v)}\sum\limits_{c'\in N_H(c)}\left(1 - \frac{\eta}{\ell}\right)^{-|N_H(c)\cap N_H(c')|}.
\end{align*}
We consider two separate cases depending on the degree of $c\in L(v)$ in $H$.

\smallskip

\textbf{Case 1:} $\deg_H(c) < d^{1-\theta}$, where $\theta \defeq 1/(5t)$. In this case, we have the following:
\begin{align*}
    \sum\limits_{c'\in N_H(c)}\left(1 - \frac{\eta}{\ell}\right)^{-|N_H(c)\cap N_H(c')|} &\leq \sum\limits_{c'\in N_H(c)}\left(1 - \frac{\eta}{\ell}\right)^{-d^{1-\theta}} \\
    &\leq \deg_H(c)\left(1 - \frac{\eta}{\ell}\,d^{1-\theta}\right)^{-1} \\
    &\leq \deg_H(c)\left(1 - d^{-\theta}/4\right)^{-1} \\
    &\leq \deg_H(c)\left(1 + d^{-\theta}\right).
\end{align*}

\textbf{Case 2:} $\deg_H(c) \geq d^{1-\theta}$. Here we will invoke the $K_{1,s,t}$-freeness of $H$. We say:
\begin{align*}
    \text{$c' \in N_H(c)$ is }&\text{\emphdef{bad} if $c'$ has at least $\deg_H(c)^{1-\theta}$ neighbors in $N_H(c)$,} \\
    &\text{\emphdef{good} otherwise.}
\end{align*}
Let $\Bad$ and $\Good$ be the sets of bad and good colors in $N_H(c)$ respectively. The following claim shows that $|\Bad|$ is relatively small.

\begin{claim}\label{claim:Bad}
$|\Bad| \leq \deg_H(c)^{1-\theta}$. 
\end{claim}

\begin{claimproof}
Since every bad color has at least $\deg_H(c)^{1-\theta}$ neighbors in $N_H(c)$, we have
\[|\Bad| \,\leq\, \frac{2\big| E(H[N_H(c)])\big|}{\deg_H(c)^{1-\theta}} \,\leq\, \frac{2\deg_H(c)^2}{\deg_H(c)^{1-\theta}} \,=\, 2\deg_H(c)^{1+\theta}.\]
Since $H$ is $K_{1,s,t}$-free, $H[N_H(c)]$ must be $K_{s,t}$-free. Let $B$ be the bipartite graph with $V(B) = X \sqcup Y$ where $X$ is a copy of $\Bad$, and $Y$ is a copy of $N_H(c)$; place an edge between $x \in X$ and $y \in Y$ if and only if $x$ and $y$ correspond to the endpoints of an edge in $N_H(c)$.
Note that if a color $c' \in N_H(c)$ is bad, then $B$ contains two copies of $c'$: one in $X$ and the other in $Y$. However, these two copies cannot be adjacent to each other in $B$. Thus, for any $c' \in N_H(c)$, if a subgraph of $B$ is isomorphic to $K_{s,t}$, then this subgraph can contain at most one copy of $c'$. Since $H[N_H(c)]$ is $K_{s,t}$-free, we conclude that $B$ is $K_{s,t}$-free as well. If we set $\hat{\beta} = 1/t = 5\theta$, $m = |\Bad|$, $n = \deg_H(c)$, then, by the \hyperref[KST]{K\H{o}v\'ari--Sós--Turán theorem} \ep{Theorem~\ref{KST}},
\[|E(B)| \,\leq\, s^{\hat{\beta}}|\Bad|^{1-\hat{\beta}}\deg_H(c) + t|\Bad| \leq 4s^{\hat{\beta}}\deg_H(c)^{2+\theta-\hat{\beta}-\hat{\beta}\theta}.\]
Furthermore, by the definition of $\Bad$, we have
\[|E(B)| \,\geq\, |\Bad|\deg_H(c)^{1 - \theta}.\]
Putting together the previous inequalities, we get
\[|\Bad| \leq 4s^{\hat{\beta}}\deg_H(c)^{1 + 2\theta - \hat{\beta}}.\]
Finally, note that for $d$ large enough, we have 
\[s \,\leq\, d^{1/10} \,\leq\, d^{(1-\theta)/5} \,\leq\, \deg_H(c)^{1/5},\]
which implies
\[|\Bad| \,\leq\, 4\deg_H(c)^{1-2\theta} \,\leq\, \deg_H(c)^{1-\theta},\]
as desired.
\end{claimproof}

With Claim~\ref{claim:Bad} in hand, we can complete our analysis for Case 2. We write
\begin{align*}
    \sum\limits_{c'\in N_H(c)}\left(1 - \frac{\eta}{\ell}\right)^{-|N_H(c)\cap N_H(c')|} &= \sum\limits_{c'\in \Bad}\left(1 - \frac{\eta}{\ell}\right)^{-|N_H(c)\cap N_H(c')|} + \sum\limits_{c'\in \Good}\left(1 - \frac{\eta}{\ell}\right)^{-|N_H(c)\cap N_H(c')|} \\
    &\leq \sum\limits_{c'\in \Bad}\left(1 - \frac{\eta}{\ell}\right)^{-\deg_H(c)} + \sum\limits_{c'\in \Good}\left(1 - \frac{\eta}{\ell}\right)^{-\deg_H(c)^{1-\theta}} \\
    &\leq |\Bad|\left(1 - \frac{\eta}{\ell}\right)^{-2d} + (\deg_H(c) - |\Bad|)\left(1 - \frac{\eta}{\ell}\right)^{-\deg_H(c)^{1-\theta}}.
\end{align*}
The last expression increases with $|\Bad|$ and so we may apply the upper bound on $|\Bad|$ from Claim~\ref{claim:Bad}. Furthermore, $\left(1 - \frac{\eta}{\ell}\right)^{-2d} = 1/\keep \leq 2$ by \ref{item:ell}, so we now have
\begin{align*}
    \sum\limits_{c'\in N_H(c)}\left(1 - \frac{\eta}{\ell}\right)^{-|N_H(c)\cap N_H(c')|} &\leq 2\deg_H(c)^{1-\theta} + (\deg_H(c) - \deg_H(c)^{1-\theta})\left(1 - \frac{\eta}{\ell}\deg_H(c)^{1-\theta}\right)^{-1} \\
    &\leq \deg_H(c)\left(2\deg_H(c)^{-\theta} + \dfrac{1 - \deg_H(c)^{-\theta}}{1 - \frac{\eta}{\ell}\,2d\,\deg_H(c)^{-\theta}}\right) \\
    &\leq \deg_H(c)(d^{-\theta/2} + 1).
\end{align*}
Putting both cases together and since $\eta\beta > d^{-\theta/2}$ by \eqref{etabeta}, we conclude that:
\begin{align*}
    \E[|\keptedges|] &\leq \keep^2\left(\sum\limits_{\text{Case 1}}\deg_H(c)(1 + d^{-\theta}) + \sum\limits_{\text{Case 2}}\deg_H(c)(1+d^{-\theta/2})\right) \\
    &\leq \keep^2\sum\limits_{c\in L(v)}\deg_H(c)(1 + \eta\beta) \\
    &= \keep^2\,\ell(v)\,\overline{\deg}(v)(1+\eta\beta). \qedhere
\end{align*}
\end{proof}

    Next, we bound $\E[|\keptedges\cap\uncoloredges|]$ in terms of $\E[|\keptedges|]$. Recall that
    \[
        \uncoloredges = \{cc'\in E(H)\,:\, c\in L(v),\ c'\in N_H(c), \text{ and } c' \in U\}.
    \]

\begin{Lemma}\label{keptUncolorExpectation}
$\E[|\keptedges\cap\uncoloredges|] \leq \uncolor\,(1+4\eta\beta)\E[|\keptedges|]$.
\end{Lemma}

\begin{proof}\stepcounter{ForClaims} \renewcommand{\theForClaims}{\ref{keptUncolorExpectation}}
    Note the following:
    \begin{align*}
    \E[|\keptedges\cap\uncoloredges|] &= \sum\limits_{c\in L(v)}\sum\limits_{c'\in N_H(c)}\P[c,c'\in K,\, c' \in U] \\
    &= \sum\limits_{c\in L(v)}\sum\limits_{c'\in N_H(c)}\P[c,c'\in K]\,\P[c'\in U\,|\,c,c'\in K].
    \end{align*}
It remains to show that for all $c \in L(v)$ and $c' \in N_H(c)$,
\[
    \P[c'\in U\,|\,c,c'\in K] \,\leq\, \uncolor(1+4\eta\beta).
\]
    To this end, take any $c \in L(v)$ and $c' \in N_H(c)$ and let
\[u = L^{-1}(c'), \qquad \psi(u) = (1-\eta\beta)\keep\,\ell(u).\] 
Then we can write
\begin{align}
    &\P[c'\in U \,|\, c,c'\in K] \nonumber\\
    =\, &\P[\phi(u) = \blank, \, k(u) \leq  \psi(u)\,|\, c,c'\in K] + \P[\phi(u) = \blank,\, k(u) > \psi(u) \,|\, c,c'\in K] \nonumber\\
    \leq\, &\P[k(u) \leq  \psi(u)\,|\, c,c'\in K] + \P[\phi(u) = \blank,\, k(u) > \psi(u) \,|\, c,c'\in K].\label{eq:two_terms}
\end{align}
We bound the first term with the following claim.
\begin{claim}\label{claim:ksmall}
$\P[k(u)\leq \psi(u)\,|\, c,c'\in K] \leq \exp\left(-d^{1/10}\right)$.
\end{claim}
\begin{claimproof}
    We use \hyperref[harris]{Harris's Inequality} (Theorem \ref{harris}). For each $\tilde{c}
\in V(H)$, define the following events:
\begin{align*}
    \tilde{c}_a&\defeq \set{\tilde{c} \text{ is activated}}, \\
    \tilde{c}_e&\defeq \set{\tilde{c} \text{ fails its equalizing coin-flip, i.e., } \eq(\tilde{c}) = 0}.
\end{align*}
Let $\Gamma = \{\tilde{c}_a, \tilde{c}_e \,:\, \tilde{c} \in V(H)\}$ and let $S \subseteq \Gamma$ be the random set of events in $\Gamma$ that took place during the execution of the \hyperlink{procedure}{Coloring Procedure}. Note that $S$ is formed randomly by including each event independently with probability
\[
    p(\tilde{c}_a) = \frac{\eta}{\ell} \qquad \text{and} \qquad p(\tilde{c}_e) = 1 - \frac{\keep}{(1-\eta/\ell)^{\deg_H(\tilde{c})}}.
\]
    Since the set $S$ encodes all the random choices made in the \hyperlink{procedure}{Coloring Procedure}, the outcome of the procedure can be computed from the value of $S$. In particular, for any subset $I \subseteq \Gamma$, we may let $K_I \subseteq V(H)$ denote the set such that $K = K_I$ when $S = I$. Explicitly,
    \[
        K_I \,=\, \set{\tilde{c} \in V(H) \,:\, \tilde{c}_e \not \in I \text{ and } \hat{c}_a \not \in I \text{ for all } \hat{c} \in N_H(\tilde{c})}.
    \]
    It follows that $K_I \subseteq K_J$ whenever $I \supseteq J$. Now we define the following two families of subsets of $\Gamma$:
\begin{align*}
    \mathcal{A}&\defeq \{I \subseteq \Gamma \,:\, \text{when $S = I$, we have $k(u)\leq \psi(u)$}\}, \\
     \mathcal{B}&\defeq \{I \subseteq \Gamma \,:\, \text{when $S = I$, we have $c$, $c' 
     \in K$}\}.
\end{align*}
    The monotonicity of $K_I$ implies that
    $\mathcal{A}$ is an increasing family of subsets of $\Gamma$, and $\mathcal{B}$ is a decreasing family.
    Hence, by \hyperref[harris]{Harris's Inequality}, $\P[S \in \mathcal{A} \,\vert\, S \in  \mathcal{B}] \leq \P[S \in \mathcal{A}]$. Using Lemma \ref{listConcentration} we obtain
\begin{align*}
    \P[k(u)\leq \psi(u)\,|\, c,c'\in K] \,\leq\, \P[k(u)\leq \psi(u)] \,\leq\, \exp\left(-d^{1/10}\right),
\end{align*}
    as claimed.
    \end{claimproof}

Let us now move our attention to the second term in \eqref{eq:two_terms}. We have:
\begin{align*}
    &\P[\phi(u) = \blank, \,k(u) > \psi(u)\,|\, c,c'\in K] \\
    =\, &\P[\phi(u) = \blank \,|\,k(u) > \psi(u) \text{ and } c,c'\in K]\ \P[k(u) > \psi(u)\,|\, c,c'\in K]\\
    \leq\,&\P[\phi(u) = \blank \,|\,k(u) > \psi(u) \text{ and } c,c'\in K].
\end{align*}
Since we are conditioning on $c \in K$, we know that $c'\notin A$. We are also conditioning on $k(u) > \psi(u)$ and so $\phi(u)=\blank$ if and only if $k(u) - 1 \geq \psi(u) - 1$ colors in $L(u)$ are not activated. (Here the ``$-1$'' comes from the fact that we know $c'\notin A$; note that $c'$ is the only neighbor of $c$ in $L(u)$, so we do not know anything about the activation status of the other colors in $L(u)$.) Therefore, we have,
\begin{align*}
    \P[\phi(u) = \blank \,|\,k(u) > \psi(u) \text{ and } c,c'\in K] \,\leq\, \left(1-\frac{\eta}{\ell}\right)^{\psi(u)}\,\left(1+\frac{\eta}{\ell}\right) \,\leq\, \left(1-\frac{\eta}{\ell}\right)^{\psi(u)}(1+1/(4d)),
\end{align*}
where the last inequality follows from \ref{item:ell}.
Observe the following:
\begin{align*}
    \psi(u) \,\geq\, (1-\eta\beta)(1-\beta)\keep\,\ell/2 \,\geq\, (1-(1+\eta)\beta)\keep\,\ell/2.
\end{align*}
With this bound, it follows that
\begin{align*}
    \left(1-\frac{\eta}{\ell}\right)^{\psi(u)}(1+1/(4d)) \, &\leq\, \left(1-\frac{\eta}{\ell}\right)^{(1-(1+\eta)\beta\,\keep\,\ell/2}(1+1/(4d)) \\
    &=\,\uncolor\left(1-\frac{\eta}{\ell}\right)^{-(1+\eta)\beta\,\keep\,\ell/2}(1+1/(4d)) \\
    &\leq\, \uncolor(1+\eta(1+\eta)\beta)(1+1/(4d)).
\end{align*}
Combining this with \eqref{eq:two_terms} and Claim~\ref{claim:ksmall}, we obtain the desired bound for large enough $d$:
\begin{align*}
    \P[\phi(u)=\blank\,|\,c,c'\in K] &\leq \exp\left(-d^{1/10}\right)+\uncolor(1+\eta(1+\eta)\beta)(1+1/(4d)) \\
    &\leq \uncolor(1 + 4\eta\beta).\qedhere
\end{align*}
\end{proof}
Lemmas~\ref{keptExpectation} and \ref{keptUncolorExpectation} together imply that
\begin{align*}
    \E[|\keptedges\cap\uncoloredges|] &\leq \keep^2\,\uncolor\,\ell(v)\,\overline{\deg}(v)(1+4\eta\beta)(1+\eta\beta) \\
    &\leq \keep^2\,\uncolor\,\ell(v)\,\overline{\deg}(v)(1+8\eta\beta),
\end{align*}
which completes the proof of Lemma \ref{expectationKeptUncolor}.

\section{Proof of Lemma \ref{concentrationKeptUncolor}}\label{sectionProofofConcentration}
In this section, we prove Lemma \ref{concentrationKeptUncolor}, which we restate below for convenience (see \S \ref{sectionProofOfIterationTheorem} for the definitions of $E_K(u)$ and $E_K(v)$):

\begin{Lemma*}[\ref{concentrationKeptUncolor}]
$\P[|\keptedges\cap\uncoloredges| \geq \keep^2\,\uncolor\,\ell(v)(\overline{\deg}(v) + 8\eta\beta\,d_{\max})] \leq d^{-100}$.
\end{Lemma*}

    Instead of working with the quantity $|\keptedges\cap \uncoloredges|$ directly, we focus separately on the two values $|\uncoloredges|$ and $|\uncoloredges\setminus\keptedges|$, both of which turn out to be easier to bound. We will show that, with sufficiently high probability, the former is not too much larger than its expected value and the latter is not too much smaller than its expected value. The identity $|\keptedges\cap \uncoloredges| = |\uncoloredges| - |\uncoloredges\setminus\keptedges|$ will then allow us to achieve the desired upper bound on $|\keptedges\cap \uncoloredges|$. 

    To derive the desired concentration of measure results, we shall use \hyperref[ExceptionalTalagrand]{Exceptional Talagrand's inequality} \ep{Theorem \ref{ExceptionalTalagrand}}. In both instances, we will use the same set of independent trials. Namely, our trials will be the equalizing coin flips together with a trial for each vertex $u \in V(G)$ that contains the information about which colors in $L(u)$ are activated. 
    More precisely, for a color $c \in V(H)$, let $\mathsf{act}(c)$ be the indicator random variable of the event $\set{\text{$c$ is activated}}$. For a vertex $u \in V(G)$, we write $L(u)$ as $L(u) = \{c_1, \ldots, c_{\ell(u)}\}$ and define
    \[
        \mathsf{act}(u) \,\defeq\, (\mathsf{act}(c_1),\ldots, \mathsf{act}(c_{\ell(u)})) \,\in\, \set{0,1}^{\ell(u)}.
    \]
    Then $\mathcal{T} \defeq \set{\mathsf{act}(u),\, \eq(c)\,:\, u \in V(G),\, c \in V(H)}$ is a family of independent random trials that completely determine the output of the \hyperlink{procedure}{Coloring Procedure}. Let $\Omega$ be the set of all possible outcomes of these trials. Set $C \defeq 20$ and let $\Omega^* \subseteq \Omega$ be the set of all outcomes under which at least one of the following statements holds:
    \begin{enumerate}[label=\ep{\normalfont{}Ex\arabic*}]
        \item\label{item:Ex1} there is a color $c \in L(N_G^2[v])$ such that $|N_H(c)\cap A| \geq C\log d$, or
        \item\label{item:Ex2} there is $u \in N_G^2[v]$ such that $|A(u)| \geq C\log d$.
    \end{enumerate}
    Here and in what follows, $N^2_G[v]$ is the set of all vertices at distance at most $2$ from $v$.

\begin{Lemma}\label{omegaStar}
    $\P[\Omega^*] \leq 2d^{-125}$.
\end{Lemma}

\begin{proof}
    For any color $c$, the number of colors in $N_H(c) \cap A$ is a binomial random variable with at most $2d$ trials, each with probability $\eta/\ell$. 
    By the union bound, 
\begin{align*}
    \P\left[|N_H(c) \cap A| \geq C{\log d}\right] &\leq {2d \choose \lceil C \log d \rceil} \left(\frac{\eta}{\ell}\right)^{\lceil C \log d \rceil} \leq \left(\frac{2ed}{\lceil C \log d \rceil}\right)^{\lceil C \log d \rceil}  \left(\frac{\eta}{\ell}\right)^{\lceil C \log d \rceil} \\
    &\leq \left(\frac{e}{2\lceil C \log d \rceil}\right)^{\lceil C \log d \rceil} \leq d^{-150},
\end{align*}
where the last inequality holds for $d$ large enough.
By the union bound and the fact that $\deg_G(v) \leq 2(1+\beta)\ell d \leq 220d^2$, we get
\begin{align*}
    \P\left[\text{\ref{item:Ex1} holds}\right] \,\leq\, |L(N_G^2[v])| d^{-150} \,\leq\, d^{-125},
\end{align*}
for $d$ large enough. Similarly, for any $u \in N^2_G[v]$, the number of colors in $A(u)$ is a binomial random variable with $\ell(u)$ trials, each having probability $\eta/\ell$. 
By the union bound,
\begin{align*}
    \P[|A(u)| \geq C\log d] &\leq {\ell(u) \choose \lceil C \log d \rceil} \left(\frac{\eta}{\ell}\right)^{\lceil C \log d \rceil} \leq \left(\frac{e\ell(u)}{\lceil C \log d \rceil}\right)^{\lceil C \log d \rceil}  \left(\frac{\eta}{\ell}\right)^{\lceil C \log d \rceil} \\
    &\leq \left(\frac{e}{\lceil C \log d \rceil}\right)^{\lceil C \log d \rceil} \leq d^{-150},
\end{align*}
where the last inequality holds for $d$ large enough.
Since $\ell(u) \leq (1+\beta)\ell \leq 110d$, we get
\begin{align*}
    \P\left[\text{\ref{item:Ex2} holds}\right] \leq |N_G^2[v]| d^{-150} \leq d^{-125}.
\end{align*}
Thus $\P\left[\Omega^*\right] \leq 2d^{-125}$, as claimed.
\end{proof}

    \hypertarget{certificate}{}\hypertarget{uncolorcertificate}{} For each removed color $c \not\in K$, we define a trial $\mathsf{cert}(c) \in \mathcal{T}$ as follows:
    \begin{itemize}
        \item If $\eq(c) = 0$, we let $\mathsf{cert}(c) \defeq \eq(c)$.
        
        \item Otherwise, there must exist a vertex $w \in V(G)$ such that $N_H(c) \cap A(w) \neq \0$. In this case we let $\mathsf{cert}(c) \defeq \mathsf{act}(w)$ for any such vertex $w$.
    \end{itemize}
    We call $\mathsf{cert}(c)$ the \emphd{removal certificate} of $c$. Crucially, if $c \not\in K$, this fact is certified by the outcome of the trial $\mathsf{cert}(c)$, in the sense that, given the current outcome of $\mathsf{cert}(c)$, $c$ would not be in $K$ regardless of the outcomes of the other trials. For an uncolored vertex $u \in V(G)$, every activated color in $L(u)$ is removed, so we can define the set of trials
    \[
        \mathsf{cert}(u) \,\defeq\, \set{\mathsf{act}(u)} \cup \set{\mathsf{cert}(c) \,:\, c \in A(u)}.
    \]
    We call $\mathsf{cert}(u)$ the \emphd{uncoloring certificate} of $u$, because if $u$ is uncolored, this fact is certified by the outcomes of the trials in $\mathsf{cert}(u)$. Notice that, by definition, $|\mathsf{cert}(u)| \leq 1 + |A(u)|$. 

For the following lemmas, recall that $d_{\max} \defeq \max\{d^{7/8}, \Delta(H)\}$. 

\begin{Lemma}\label{uncolorGconcentration}
$\P\left[|\uncoloredges| > \E\left[|\uncoloredges|\right]+ \eta\beta\,\keep^2\,\uncolor\,\ell(v)\,d_{\max}\right] \leq d^{-105}$.
\end{Lemma}
\begin{proof}\stepcounter{ForClaims} \renewcommand{\theForClaims}{\ref{uncolorGconcentration}}
    Instead of bounding $|\uncoloredges|$ directly, we partition $\uncoloredges$ by defining, for each $c \in L(v)$,
\begin{align*}
    \uncoloredgesc \defeq \{cc'\in E(H)\,:\, c' \in U\}.
\end{align*}
We now use \hyperref[ExceptionalTalagrand]{Exceptional Talagrand's inequality} to prove the concentration of $|\uncoloredgesc|$.

\begin{claim}\label{uncolorHconcentration}
    $\P\left[\big||\uncoloredgesc|-\E\left[|\uncoloredgesc|\right]\big| > \eta\beta\, \keep^2\,\uncolor\, d_{\max}\right] \leq d^{-110}$ for every $c \in L(v)$.
\end{claim}
\begin{claimproof}
    We claim that the random variable $|\uncoloredgesc|$ satisfies conditions \ref{item:ET1} and \ref{item:ET2} of Theorem~\ref{ExceptionalTalagrand} with $s = (1 + C \log d)\deg_H(c)$, $\gamma = 1 + (C{\log d})^2$, and the exceptional set $\Omega^\ast$.
    
    To verify \ref{item:ET1}, take any $\omega \in \Omega \setminus \Omega^*$ and consider the output of the \hyperlink{procedure}{Coloring Procedure} under $\omega$.  Let $U^* \subseteq N_G(v)$ be the set of all neighbors of $v$ whose list of colors contains an endpoint of an edge in $E_U(v,c)$. That is, $U^*$ is defined as follows:
    \[
        U^\ast \,\defeq\, \set{u \in N_G(v) \,:\, \phi(u) = \blank \text{ and } N_H(c) \cap L(u) \neq \0}.
    \]
    Since every vertex $u \in U^\ast$ is uncolored, we can form a set $I$ of trials as follows:
    \[
        I \,\defeq\, \bigcup_{u \in U^\ast} \mathsf{cert}(u).
    \]
    \ep{Recall that for an uncolored vertex $u$, $\mathsf{cert}(u)$ is the \hyperlink{uncolorcertificate}{uncoloring certificate} of $u$.} 
    Since $|U^\ast| \leq \deg_H(c)$ and $|A(u)| < C\log d$ for all $u \in U^\ast$ \ep{because $\omega \not \in \Omega^\ast$}, we have
    \[
        |I| \,\leq\, (1 + C \log d) \deg_H(c) \,=\, s.
    \]

    Now take any $q > 0$ and suppose that a set of outcomes $\omega' \in \Omega \setminus \Omega^*$ satisfies \[|\uncoloredgesc(\omega')| \,\leq\, |\uncoloredgesc(\omega)| - q.\] For each vertex $u \in U^\ast$ that is an endpoint of an edge in $\uncoloredgesc(\omega)\setminus \uncoloredgesc(\omega')$, the outcomes of at least one trial in $\mathsf{cert}(u)$ must be different in $\omega$ and in $\omega'$. A trial of the form $\eq(c')$ can belong to at most one set $\mathsf{cert}(u)$ \ep{namely the one corresponding to $u = L^{-1}(c')$}. Furthermore, if some $u \in U^\ast$ and $w \in V(G)$ satisfy $\mathsf{act}(w) \in \mathsf{cert}(u)$, then $w \in N_G[u]$ and either $w = u$ or
    \begin{equation}\label{eq:uw}
        N_H(A(w)) \cap A(u) \,\neq\, \0.
    \end{equation}
    Since $\omega \not \in \Omega^\ast$, each vertex $w \in N_G^2[v]$ has at most $(C \log d)^2$ neighbors $u$ such that \eqref{eq:uw} holds. It follows that every trial in $\mathcal{T}$ can belong to at most $1 + (C{\log d})^2 = \gamma$ sets $\mathsf{cert}(u)$ corresponding to $u \in U^\ast$. Therefore, $\omega'$ and $\omega$ must differ on at least $q/\gamma$ trials in $I$, as desired.

    It remains to show that $\P\left[\Omega^*\right] \leq M^{-2}$, where $M = \max\{\sup |\uncoloredgesc|, 1\}$. Note that \[\max\{\sup |\uncoloredgesc|, 1 \} \leq \deg_H(c) \leq 2d,\]
so it follows from Lemma \ref{omegaStar} that $\P\left[\Omega^*\right] \leq 1/M^2$, for $d$ large enough.

We can now use \hyperref[ExceptionalTalagrand]{Exceptional Talagrand's inequality}. Let $\xi \defeq \eta\beta\,\keep^2\,\uncolor\,d_{\max}$. Note that $\xi > 50\gamma\sqrt{s}$ \ep{if $d$ is large enough}. 
We can therefore write 
\begin{align*}
&\P\left[\big||\uncoloredgesc|-\E\left[|\uncoloredgesc|\right]\big| > \eta\beta \,\keep^2\,\uncolor\,d_{\max}\right] \\
\leq\,& 4\exp{\left(-\frac{\eta^2\beta^2\,\keep^4\,\uncolor ^2\,d_{\max}^2}{16(1 + (C{\log d})^2)^2(1 + C \log d) \deg_H(c))}\right)} + 4\P[\Omega^*]\\
\leq\,& 4\exp\left( -\Omega\left(\frac{\eta^2\beta^2 d^{3/4}}{\log^5 d}\right)\right) + 8d^{-125} \\
\leq\,& d^{-110},
\end{align*}
for $d$ large enough.
\end{claimproof}

Putting all of this together, we have:
\begin{align*}
    &\P\left[|\uncoloredges| > \E[|\uncoloredges|] + \eta\beta\,\keep^2\,\uncolor\,\ell(v)\,d_{\max}\right] \\
    \leq\, &\P\left[\exists~ c\in L(v) \text{ such that } |\uncoloredgesc| > \E[|\uncoloredgesc|] + \eta\beta\,\keep^2\,\uncolor\, d_{\max}\right] \\
    \leq\, &\ell(v)\,d^{-110} \\
    \leq\, &d^{-105}. \qedhere
\end{align*}
\end{proof}

Now, we turn our attention toward $|\uncoloredges\setminus\keptedges|$. Let $S \defeq E_H[L(v), N_H(L(v))]$ and $\mu \defeq 1/6$.

\begin{Lemma}\label{lemma:smallS}
    If $|S| \leq d^{2-\mu}$, then Lemma~\ref{concentrationKeptUncolor} holds.
\end{Lemma}
\begin{proof}
    Assume that $|S| \leq d^{2- \mu}$. Since $\keep$, $\uncolor \geq \frac{1}{2}$, it follows that
\begin{align*}
    |S| \,\leq\, d^{2- \mu} \,<\, \frac{8}{7}d^{1+7/8-1/(50t)} \,\leq\, \frac{8\keep^2\uncolor}{1-\keep^2\uncolor}d^{1-1/(50t)}d_{\max}.
\end{align*}
Using that $\ell(v) \geq \frac{d}{\log^5(d)} \geq d^{1 - 1/(100t)}$ for $d$ large enough, and that $\ell(v)\overline{\deg}(v) = |S|$, we get
\begin{align*}
    |S| &< \frac{8\eta\beta\, \keep^2\uncolor\ell(v)d_{\max}}{1-\keep^2\uncolor}
\end{align*}
from which it follows that
\begin{align*}
     |S| &< \keep^2\uncolor|S| + 8\eta\beta \keep^2\uncolor\ell(v)d_{\max} \\
    &\leq \keep^2\uncolor\ell(v)\overline{\deg}(v) + 8\eta\beta \keep^2\uncolor\ell(v)d_{\max}.
\end{align*}
Therefore, $|\uncoloredges\cap\keptedges| \leq |S| < \keep^2\uncolor\ell(v)(\overline{\deg}(v) + 8\eta\beta\,d_{\max})$, as desired. 
\end{proof}

    In view of Lemma~\ref{lemma:smallS}, we may now assume that $|S|>2^{d-\mu}$ for the rest of the proof. This assumption allows us to establish a strong concentration bound on $|\uncoloredges\setminus\keptedges|$. 

\begin{Lemma}\label{UminusKconcentration}
    If $|S| > d^{2-\mu}$, then \[\P[|\uncoloredges\setminus\keptedges| \leq \E[|\uncoloredges\setminus\keptedges|] - \eta\beta\,\keep^2\uncolor\ell(v)d_{\max}] \,\leq\, d^{-105}.\]
\end{Lemma}
\begin{proof}\stepcounter{ForClaims} \renewcommand{\theForClaims}{\ref{UminusKconcentration}}
    We make the following definitions for $\tau \defeq 1/5$, $c'\in V(H)$, and $u \in N_G(v)$:
\begin{align*} 
k &\defeq \lceil d^{\tau} \rceil,  \qquad \Gamma(c') \defeq \{xy \in S \,:\, c'x \in E(H) \text{ or } c'y \in E(H)\}, \qquad M_u \defeq E_H[L(v), L(u)].
\end{align*}
\begin{claim}[Random partitioning]\label{partition}
There exists a partition of $S$ into sets $S_1$, \ldots, $S_k$ such that the following statements hold for every $1\leq i \leq k$, $c'\in L(N_G^2[v])$, and $u \in N_G(v)$:
\begin{enumerate}[label=\ep{\normalfont{}P\arabic*}]
    \item\label{item:P1} $\frac{d^{2-\tau-\mu}}{2} \leq |S_i| \leq 400d^{2-\tau}$,
    
    \item\label{item:P2} $|M_u\cap S_i| \leq \frac{3}{2}\,\ell(u)\,d^{-\tau}$,
    
    \item\label{item:P3} $|\Gamma(c') \cap S_i| \leq 6d^{1-\tau}$.
\end{enumerate}
\end{claim}

\begin{claimproof}
    Form a partition $S= S_1 \sqcup \ldots \sqcup S_{k}$ by independently placing each edge in $S$ into one of the parts uniformly at random. Since \[d^{2-\mu} \,<\, |S| \,=\, \ell(v) \overline{\deg}(v) \,\leq\, 2(1+\beta)\ell d \,\leq\, 300d^{2}\] and $d^\tau \leq k \leq 4d^\tau/3$, we have
    \[
        \frac{3 d^{2-\tau-\mu}}{4} \,\leq\, \E[|S_i|] \,\leq\, 300d^{2-\tau}.
    \]
    By the \hyperref[chernoff]{Chernoff bound} (Theorem \ref{chernoff}), 
    \[
        \P\left[|S_i| \notin \left(1 \pm \frac{1}{3}\right)\E[|S_i|] \right] \,\leq\, 2\exp\left({-\frac{\E[|S_i|]}{27}}\right) \,\leq\, 2\exp\left({-\frac{d^{2-\tau-\mu}}{36} }\right).
    \]
Therefore, 
\begin{align}\label{p1}
   \P\left[\exists\, i \text{ such that } |S_i| \notin \left[(1/2) d^{2-\tau-\mu},\, 400d^{2-\tau}\right] \right] \leq 2k\exp\left({-\frac{d^{2-\tau -\mu}}{36}}\right) \leq d^{-1}.
\end{align}

    Next, for each $u \in N_G(v)$, let $m_i(u) \defeq |M_u\cap S_i|$. Then $m_i(u)$ is a binomial random variable with at most $\ell(u)$ trials with success probability $1/k$. We may assume that $|M_u| \geq \ell(u)d^{-\tau}$ (otherwise \ref{item:P2} holds trivially) and hence $\E[m_i(u)] \in [(3/4)\ell(u)d^{-2\tau}, \ell(u)d^{-\tau}]$. By the \hyperref[chernoff]{Chernoff bound} again,
\begin{align*}
    \P\left[m_i(u) \geq \left(1+\frac 12\right)\ell(u)d^{-\tau}]\right] &\leq 2\exp\left(-\frac{\E[m_i(u)]}{12}\right) \leq 2\exp\left(-\frac{d^{1/2-2\tau}}{16}\right).
\end{align*}
The last inequality follows since $\ell(u) \geq (1-\beta)\ell/2 \geq 2(1-\beta)d/(\log d)^5 \geq d^{1/2}$ for $d$ large enough. By the union bound, we may conclude that
\begin{align}\label{p2}
    \P\left[\exists\, u \in N_G(v), \ \exists i \text{ such that } m_i(u) \geq \left(1+\frac 12\right)\ell(u)d^{-\tau}\right] \leq 4dk\exp\left(-\frac{d^{1/2-2\tau}}{16}\right) \leq d^{-1}.
\end{align}

    Finally, consider any $c' \in L(N^2_G[v])$. If $|\Gamma(c')| \leq d^{1-\tau}$, then \ref{item:P3} holds trivially, so we may assume that $|\Gamma(c')| \geq d^{1-\tau}$. We now show that $|\Gamma(c')| \leq 4d$. Since the edges between two lists in $H$ form a matching, it follows that $c'$ can have at most one neighbor in $L(v)$, which is incident to at most $2d$ edges in $S$. Also, $c'$ can have at most $2d$ neighbors in $N_H(L(v))$, each of which is incident to a single edge in $S$. Hence, the desired bound on $|\Gamma(c')|$ follows. Now let $r_i(c') \defeq |\Gamma(c') \cap S_i|$. This is a binomial random variable, and we have
    \[\frac{3d^{1-2\tau}}{4} \,\leq\, \E[r_i(c')] \,\leq\, 4d^{1-\tau}.\]
    By the \hyperref[chernoff]{Chernoff bound}, we find that
\[\P\Big[r_i(c') \geq \left(1+\frac{1}{2}\right)\E[r_i(c')]\Big] \leq 2\exp\left(-\frac{\E[r_i(c')]}{12}\right) \leq 2\exp\left(-\frac{d^{1 - 2\tau}}{16}\right).\]
By the union bound,
\begin{align}
    \P\Big[\exists\, c'\in L(N^2_G[v]),\ \exists\, i \text{ such that } r_i(c') \geq \left(1+\frac{1}{2}\right)\E[r_i(c')]\Big] \,&\leq\, 20(1+\beta)\ell d^2 k \exp\left(-\frac{d^{1 - 2\tau}}{16}\right) \nonumber\\
    &\leq\,  d^{-1}.\label{p3}
\end{align}
Putting together inequalities \eqref{p1}, \eqref{p2}, and \eqref{p3}, we see that
\[\P\left[S_1,\ldots,S_k\text{ satisfy \ref{item:P1}, \ref{item:P2}, and \ref{item:P3}}\right] \,\geq\, 1 - 3d^{-1} \,>\, 0.\]
Therefore, such a partition exists.
\end{claimproof}

From here on out, we fix a partition $S = S_1 \sqcup \ldots \sqcup S_k$ satisfying the properties in Claim \ref{partition}. Let $Z_i \defeq S_i \cap (\uncoloredges \setminus \keptedges)$. We now show that each random variable $|Z_i|$ is highly concentrated around its expected value. \ep{This is the \hyperlink{partitioning}{random partitioning technique} advertised in \S\ref{subsec:sketch}.} 

\begin{claim}\label{UminsKPartitionConcentration}
$\P\left[\big||Z_i| - \E[|Z_i|]\big| > \eta\beta\,\keep^2\,\uncolor\, |S_i|\right] \leq d^{-110}$ for all $1 \leq i \leq k$.
\end{claim}
\begin{claimproof}
    We claim that $|Z_i|$ satisfies conditions \ref{item:ET1} and \ref{item:ET2} of \hyperref[ExceptionalTalagrand]{Exceptional Talagrand's inequality} \ep{Theorem \ref{ExceptionalTalagrand}} with $s = (2 + C\log d)|S_i|$, $\gamma = 20(1 + (C \log d)^2)d^{1-\tau}$, and the exceptional set $\Omega^\ast$ defined in the beginning of \hyperref[sectionProofofConcentration]{this section}.

    To verify \ref{item:ET1}, take any $\omega \in \Omega \setminus \Omega^*$ and consider the output of the \hyperlink{procedure}{Coloring Procedure} under $\omega$. Given an edge $cc' \in Z_i$, we assemble a set $I_{cc'} \subseteq \mathcal{T}$ of trials that certifies that $cc' \in  \uncoloredges \setminus \keptedges$ as follows. For concreteness, let $c \in L(v)$ and $c' \in L(u)$ for some $u \in N_G(v)$. The statement that $cc' \in \uncoloredges \setminus \keptedges$ means that:
    \begin{itemize}
        \item $u$ is uncolored, and
        
        \item at least one color in $\set{c, c'}$ is not kept.
    \end{itemize}
    Let $c^\ast \in \set{c, c'}$ be any color that is not kept and define
    \[
        I_{cc'} \,\defeq\, \mathsf{cert}(u) \cup \set{\mathsf{cert}(c^\ast)},
    \]
    where $\mathsf{cert}(u)$ is the \hyperlink{uncolorcertificate}{uncoloring certificate} of $u$ and $\mathsf{cert}(c^\ast)$ is the \hyperlink{certificate}{removal certificate} of $c^\ast$. Notice that $|I_{cc'}| \leq 2 + |A(u)| \leq 2 + C\log d$, where the last inequality holds since $\omega \not \in \Omega^\ast$. Now we let
    \[
        I \,\defeq\, \bigcup_{cc' \in Z_i} I_{cc'}.
    \]
    Then $|I| \leq (2 + C\log d)|Z_i| \leq (2 + C\log d)|S_i| = s$.
    
    Now take any $q > 0$ and suppose that a set of outcomes $\omega' \not\in \Omega^*$ satisfies \[|Z_i(\omega')| \,\leq\, |Z_i(\omega)| - q.\] For each edge $cc' \in Z_i(\omega) \setminus Z_i(\omega')$, the outcomes of at least one trial in $I_{cc'}$ must be different in $\omega$ and in $\omega'$. At this point, for each trial $T \in I$, we need to bound the number of edges $cc' \in Z_i$ such that $T \in I_{cc'}$. We consider two cases depending on the type of the trial $T$.
    
    \smallskip
    
    \textbf{Case 1:} $T = \eq(c'')$ for some color $c''$. Suppose that $cc' \in Z_i$ is an edge such that $\eq(c'') \in I_{cc'}$. For concreteness, say that $c \in L(v)$ and $c' \in L(u)$ for some $u \in N_G(v)$. Then
    \begin{enumerate}[label=\ep{\itshape\alph*}]
        \item\label{item:3c} either $c'' \in \set{c, c'}$,
        \item\label{item:c''inAu} or $c'' \in A(u)$.
    \end{enumerate}
    By property \ref{item:P3} of the partition $S = S_1 \sqcup \ldots \sqcup S_k$, situation \ref{item:3c} occurs for at most \[|\Gamma(c'') \cap Z_i| \,\leq\, |\Gamma(c'') \cap S_i| \,\leq\, 6 d^{1 - \tau}\] edges $cc' \in Z_i$. On the other hand, in case \ref{item:c''inAu}, $u = L^{-1}(c'')$ and, by \ref{item:P2}, this happens for at most
    \[
        |M_u \cap Z_i| \,\leq\, |M_u \cap S_i| \,\leq\, \frac{3}{2}\,\ell(u)\,d^{-\tau} \,\leq\, 10d^{1-\tau}
    \]
    edges $cc' \in Z_i$. Therefore, the trial $\eq(c'')$ belongs to at most $16 d^{1 - \tau} \leq \gamma$ sets $I_{cc'}$.
    
    \smallskip
    
    \textbf{Case 2:} $T = \mathsf{act}(w)$ for some vertex $w$. Suppose that $cc' \in Z_i$ is an edge such that $\mathsf{act}(w) \in I_{cc'}$. For concreteness, say that $c \in L(v)$ and $c' \in L(u)$ for some $u \in N_G(v)$. Then
    \begin{enumerate}[label=\ep{\itshape\alph*}]
        \item\label{item:a1} either $\mathsf{act}(w) = \mathsf{cert}(c^\ast)$ for some $c^\ast \in \set{c, c'}$,
        \item\label{item:b1} or $\mathsf{act}(w) \in \mathsf{cert}(u)$.
    \end{enumerate}
    By property \ref{item:P3} of the partition $S = S_1 \sqcup \ldots \sqcup S_k$, situation \ref{item:a1} occurs for at most
    \[\sum_{c'' \in A(w)} |\Gamma(c'') \cap Z_i| \,\leq\, \sum_{c'' \in A(w)} |\Gamma(c'') \cap S_i| \,\leq\, |A(w)| \cdot 6 d^{1 - \tau} \,\leq\, 6C d^{1-\tau} \log d\] edges $cc' \in Z_i$, where in the last inequality we use that $\omega \not \in \Omega^\ast$. On the other hand, in case \ref{item:b1}, $w \in N_G[u]$ and either $w = u$ or
    \begin{equation}\label{eq:uw1}
        N_H(A(w)) \cap A(u) \,\neq\, \0.
    \end{equation}
    Since $\omega \not \in \Omega^\ast$, each vertex $w \in N_G^2[v]$ has at most $(C \log d)^2$ neighbors $u$ such that \eqref{eq:uw1} holds. Furthermore, by \ref{item:P2}, for each such $u$, there are at most
    \[
        |M_u \cap Z_i| \,\leq\, |M_u \cap S_i| \,\leq\, \frac{3}{2}\,\ell(u)\,d^{-\tau} \,\leq\, 10d^{1-\tau}
    \]
    edges $cc' \in Z_i$ with $c' \in L(u)$. Hence, the trial $\mathsf{act}(w)$ belongs to at most
    \[
        6C d^{1-\tau} \log d + 10(1 + (C \log d)^2) d^{1-\tau} \,\leq\, \gamma
    \]
    sets $I_{cc'}$.
    
    \smallskip
    
    To summarize, each trial belongs to at most $\gamma$ sets $I_{cc'}$ corresponding to the edges $cc' \in Z_i$.
Therefore, $\omega'$ and $\omega$ must differ on at least $q/\gamma$ trials, as desired.

It remains to show that $\P\left[\Omega^*\right] \leq M^{-2}$, where $M = \max\{\sup |Z_i|, 1\}$. Note that 
\[M = \max\{\sup |Z_i|, 1 \} \leq \max\{ 400 d^{2-\tau}, 1\} \leq d^2,\]
so it follows from Lemma \ref{omegaStar} that $\P\left[\Omega^*\right] \leq 1/M^2$, for $d$ large enough.

We can now use \hyperref[ExceptionalTalagrand]{Exceptional Talagrand's inequality}. Let $\xi \defeq \eta\beta\,\keep^2\uncolor |S_i|$. Note that $\xi > 50\gamma\sqrt{s}$ \ep{for $d$ large enough}. We can therefore write
\begin{align*}
&\P\left[\big||Z_i|-\E[|Z_i|]\big| > \eta\beta\,\keep^2\uncolor |S_i|\right] \\
\leq\, &4\exp{\left(-\frac{\eta^2\beta^2\keep^4\uncolor ^2|S_i|^2}{16\gamma^2s}\right)} + 4\P[\Omega^*]\\
\leq \,&4\exp\left( -\Omega\left(\frac{\eta^2\beta^2|S_i|^2}{d^{2-2\tau}|S_i|\log^5 d}\right)\right) + 8d^{-125} \\
\leq \,&4\exp\left( -\Omega\left(\frac{\eta^2\beta^2d^{\tau-\mu}}{\log^5 d}\right)\right) + 8d^{-125} \\
\leq \,&d^{-110},
\end{align*}
for $d$ large enough.\end{claimproof}

Putting all of this together, we have:
\begin{align*}
    &\P\left[|\uncoloredges\setminus\keptedges| \leq \E[|\uncoloredges\setminus\keptedges|] - \eta\beta\,\keep^2\uncolor\ell(v)d_{\max}\right] \\
    \leq\, &\P\left[|\uncoloredges\setminus\keptedges| \leq \E[|\uncoloredges\setminus\keptedges|] - \eta\beta\,\keep^2\uncolor|S|\right] \\
    \leq\,&\P\left[\exists~ i \text{ such that } |Z_i| \leq \E[|Z_i|] - \eta\beta\,\keep^2\uncolor|S_i|\right] \\
    \leq\, &d^\tau\, d^{-110} \\
    \leq\, &d^{-105}. \qedhere
\end{align*}
\end{proof}
We are now ready to bound $|\uncoloredges\cap\keptedges|$. If $|S| \leq d^{2-\mu}$, then we are done by Lemma~\ref{lemma:smallS}. Otherwise, from Lemmas \ref{uncolorGconcentration} and \ref{UminusKconcentration}, we have 
\begin{align*}
    &\P\left[|\keptedges\cap \uncoloredges| > \E[|\keptedges\cap \uncoloredges|] + 2\eta\beta\,\keep^2\uncolor\ell(v)d_{\max}\right] \\
    \leq\, &\P\left[|\uncoloredges| > \E[|\uncoloredges|] +\eta\beta\,\keep^2\uncolor\ell(v)d_{\max}\right] +  \\
    &\qquad\P\left[|\uncoloredges\setminus \keptedges| \leq \E[|\uncoloredges\setminus \keptedges|] - \eta\beta\,\keep^2\uncolor\ell(v)d_{\max}\right] \\
    \leq\, &d^{-100}.
\end{align*}
Finally, from Lemma \ref{expectationKeptUncolor} we have
\begin{align*}
    &\P\left[|\keptedges\cap \uncoloredges| > \keep^2\,\uncolor\,\ell(v)\,\overline{\deg}(v) + 8\eta\beta\,\keep^2\,\uncolor\,\ell(v)d_{\max}\right] \\
    \leq\, &\P\left[|\keptedges\cap \uncoloredges| > \keep^2\,\uncolor\,\ell(v)\,\overline{\deg}(v)(1+4\eta\beta) + 2\eta\beta\,\keep^2\,\uncolor\,\ell(v)\,d_{\max}\right] \\
    \leq\, &\P\left[|\keptedges\cap \uncoloredges| > \E[|\keptedges\cap \uncoloredges|] + 2\eta\beta\,\keep^2\,\uncolor\,\ell(v)\,d_{\max}\right] \\
    \leq\, &d^{-100}.
\end{align*}

\section{Proof of Theorem \ref{mainTheorem}}\label{sectionIterations}

    In this section we prove Theorem \ref{mainTheorem} by iteratively applying Lemma \ref{iterationTheorem} until we reach a stage where we can apply Proposition \ref{finalBlow}. To do so, we first define the parameters for the graph and the cover at each iteration and then take $d_0$ so large that the graphs at each iteration will satisfy the conditions of Lemma \ref{iterationTheorem}.

    We use the notation of Theorem~\ref{mainTheorem}. Let
    \begin{align*}
        G_1 \defeq G, \qquad \mathcal{H}_1 = (L_1, H_1) \defeq \mathcal{H},\qquad \ell_1 \defeq (4+\epsilon)d/\log d,\qquad d_1 \defeq d.
    \end{align*}
We may assume that $\epsilon$ is sufficiently small, say $\epsilon < 1/100$. Since $d$ is large, we may also assume that $\ell_1$ is an integer by slightly modifying $\epsilon$ if necessary. By removing some of the vertices from $H$ if necessary, we may assume that $|L(v)| = \ell_1$ for all $v \in V(G)$.
Define \[\kappa \defeq (2+\epsilon/4)\log(1+\epsilon/50)\approx \epsilon/25,\] and fix $\eta \defeq \kappa/\log d$, so that $\eta$ is the same each time we apply Lemma~\ref{iterationTheorem}.
We recursively define the following parameters for each $i \geq 1$:
\begin{align*}
    \keep_i &\defeq \left(1 - \frac{\eta}{\ell_i}\right)^{2d_i},
    & \uncolor_i &\defeq \left(1 - \frac{\eta}{\ell_i}\right)^{\keep_i\,\ell_i/2},
    \\
    \ell_{i+1} &\defeq \keep_i\, \ell_i, & d_{i+1}&\defeq \keep_i\, \uncolor_i\, d_i, \\
    \beta_1 &\defeq d^{-1/(200t)}, & \beta_{i+1} &\defeq \max\left\{(1+36\eta)\beta_i,\, d_{i+1}^{-1/(200t)}\right\}.
\end{align*}
Suppose that at the start of iteration $i$, the following numerical conditions hold:
\begin{enumerate}[label=\ep{\normalfont\arabic*}]
    \item\label{item:1} $d_i$ is sufficiently large: $d_i \geq \tilde{d}$, where $\tilde{d}$ is the constant from Lemma~\ref{iterationTheorem},
    \item\label{item:2} $\ell_i$ is bounded below and above in terms of $d_i$: $4\eta\,d_i < \ell_i < 100d_i$,
    \item\label{item:3} $s$ and $t$ are bounded in terms of $d_i$: $s \leq d_i^{1/10}$ and \label{item:4} $t \leq \dfrac{\tilde{\alpha}\log d_i}{\log\log d_i}$,
    \item\label{item:5} $\eta$ is close to $1/\log d_i$: $\dfrac{1}{\log^5d_i} < \eta < \dfrac{1}{\log d_i}$,
    \item\label{item:eps} $\beta_i$ is small: $\beta_i \leq 1/10$. \ep{Note that the bound $\beta_i \geq d_i^{-1/(200t)}$ holds by definition.}
\end{enumerate}
Furthermore, suppose that we have a graph $G_i$ and a DP-cover $\mathcal{H}_i = (L_i, H_i)$ of $G_i$ such that:
\begin{enumerate}[resume,label=\ep{\normalfont\arabic*}]
    \item\label{item:6} $H_i$ is $K_{1,s,t}$-free,
    \item\label{item:7} $\Delta(H_i) \leq 2d_i$,
    \item\label{item:8} the list sizes are roughly between $\ell_i/2$ and $\ell_i$: \[(1-\beta_i)\ell_i/2 \,\leq\, |L_i(v)| \,\leq\, (1+\beta_i)\ell_i \quad \text{for all } v \in V(G_i),\]
    \item\label{item:9} average color-degrees are smaller for vertices with smaller lists of colors: \[\overline{\deg}_{\mathcal{H}_i}(v) \,\leq\, \left(2 - (1 - \beta_i)\frac{\ell_i}{|L_i(v)|}\right)d_i \quad \text{for all } v\in V(G_i).\]
\end{enumerate}
Then we may apply Lemma~\ref{iterationTheorem} to obtain a partial $\mathcal{H}_i$-coloring $\phi_i$ of $G_i$ and an assignment of subsets $L_{i+1}(v) \subseteq (L_i)_{\phi_i}(v)$ to each vertex $v \in V(G_i) \setminus \dom(\phi_i)$ such that, setting
\[
    G_{i+1} \defeq G_i[V(G_i) \setminus \dom(\phi_i)], \qquad H_{i+1} \defeq H_i \left[\textstyle\bigcup_{v \in V(G_{i+1})} L_{i+1}(v)\right],
\]
\[
    \text{and} \qquad \mathcal{H}_{i+1}' \defeq (L_{i+1}', H_{i+1}'),
\]
we get that conditions \ref{item:6}--\ref{item:9} hold with $i+1$ in place of $i$.
Note that, assuming $d_0$ is large enough and $\alpha$ is small enough, conditions \ref{item:1}--\ref{item:9} are satisfied initially \ep{i.e., for $i = 1$}. Our goal is to show that there is some value $i^\star \in \N$ such that:
\begin{itemize}
    \item for all $1 \leq i < i^\star$, conditions \ref{item:1}--\ref{item:eps} hold, and
    \item we have $\ell_{i^\star} \geq 100d_{i^\star}$.
\end{itemize}
Since conditions \ref{item:6}--\ref{item:9} hold by construction, we will then be able to iteratively apply Lemma~\ref{iterationTheorem} $i^\star - 1$ times and then complete the coloring using Proposition~\ref{finalBlow}.

The following lemma finds the required $i^\star$ as well as guarantees that the list sizes remain sufficiently large to apply Proposition \ref{finalBlow}.

\begin{Lemma}\label{lemma:i_star}
The parameters $\ell_i$ and $d_i$ satisfy the following for some absolute constant $a > 0$:
\begin{enumerate}[label=\ep{\normalfont{}I\arabic*}]
    \item\label{item:10} For every $i$, $d_{i+1}/\ell_{i+1} \leq d_i/\ell_i \leq d_1/\ell_1 = \frac{\log d}{(4+\epsilon)}$.
    \item\label{item:11} For all $i$, $\ell_i \geq d^{a\,\epsilon}$.
    \item\label{item:12} There is $i^\star = O(\log d\log\log d)$ such that $d_{i^\star} \leq \ell_{i^\star}/100$.
\end{enumerate}
\end{Lemma}

\begin{proof}

Let $r_i = d_i/\ell_i$. Since $r_i$ decreases by a factor of $\uncolor_i$ at each stage, \ref{item:10} follows. We now show \ref{item:11}. First, we find a lower bound on $\keep_i$ as follows:
\begin{align*}
    \keep_i = \left(1 - \frac{\kappa}{\ell_i\,\log d}\right)^{2d_i} &\geq \exp\left(-\frac{2\kappa}{(1 - \epsilon/10)\log d}\,r_i\right) \\
    &\geq \exp\left(-\frac{2\kappa}{(1 - \epsilon/10)\log d}\,r_1\right) \\
    &= \exp\left(-\frac{2\kappa}{(1 - \epsilon/10)(4+\epsilon)}\right) \\ 
    &\geq \exp\left(-\frac{2\kappa}{4+\epsilon/2}\right).
\end{align*}
The last step follows for $\epsilon$ small enough. From this we derive and 
upper bound on $\uncolor_i$ as follows:
\begin{align*}
    \uncolor_i \,=\, \left(1 - \frac{\kappa}{\ell_i\log d}\right)^{\keep_i\,\ell_i/2} \,\leq\, \exp\left(-\frac{\keep_i\kappa}{2\log d}\right)
    \,\leq\, 1 - \frac{(1-\epsilon/10)\kappa}{2\log d}\,\exp\left(-\frac{2\kappa}{4+\epsilon/2}\right).
\end{align*}
From here, we have that:
\begin{align*}
    r_{i+1} \,=\, \uncolor_i\,r_i \,&\leq\, \left(1 - \frac{(1-\epsilon/10)\kappa}{2\log d}\,\exp\left(-\frac{2\kappa}{4+\epsilon/2}\right)\right)r_i \\
    &\leq\, \left(1 - \frac{(1-\epsilon/10)\kappa}{2\log d}\,\exp\left(-\frac{2\kappa}{4+\epsilon/2}\right)\right)^ir_1.
\end{align*}
We can use this bound on $r_i$ to achieve a better bound on $\keep_i$ as follows:
\begin{align*}
    \keep_i \,=\, \left(1 - \frac{\kappa}{\ell_i\,\log d}\right)^{2d_i} \,&\geq\, \exp\left(-\frac{2\kappa}{(1-\epsilon/10)\log d}\,r_i\right) \\
    &\geq\, \exp\left(-\frac{2\kappa}{(1-\epsilon/10)\log d}\,\left(1 - \frac{(1-\epsilon/10)\kappa}{2\log d}\,\exp\left(-\frac{2\kappa}{4+\epsilon/2}\right)\right)^{i-1}r_1\right) \\
    &=\, \exp\left(-\frac{2\kappa}{(1-\epsilon/10)(4+\epsilon)}\,\left(1 - \frac{(1-\epsilon/10)\kappa}{2\log d}\,\exp\left(-\frac{2\kappa}{4+\epsilon/2}\right)\right)^{i-1}\right) \\
    &\geq\, \exp\left(-\frac{2\kappa}{4+\epsilon/2}\,\left(1 - \frac{(1-\epsilon/10)\kappa}{2\log d}\,\exp\left(-\frac{2\kappa}{4+\epsilon/2}\right)\right)^{i-1}\right).
\end{align*}
With this and the definition of $\ell_i$, we get the desired lower bound (for $d$ large enough) as follows:
\begin{align*}
    \ell_i \,=\, \ell_1\prod\limits_{j = 1}^{i-1}\keep_j \,&\geq\, \ell_1\prod\limits_{j = 1}^{i-1}\exp\left(-\frac{2\kappa}{4+\epsilon/2}\,\left(1 - \frac{(1-\epsilon/10)\kappa}{2\log d}\,\exp\left(-\frac{2\kappa}{4+\epsilon/2}\right)\right)^{j-1}\right) \\
    &=\, \ell_1\,\exp\left(-\frac{2\kappa}{4+\epsilon/2}\,\sum\limits_{j = 1}^{i-1}\left(1 - \frac{(1-\epsilon/10)\kappa}{2\log d}\,\exp\left(-\frac{2\kappa}{4+\epsilon/2}\right)\right)^{j-1}\right) \\
    &\geq \ell_1\,\exp\left(-\frac{2\kappa}{4+\epsilon/2}\,\frac{2\log d}{(1-\epsilon/10)\kappa}\,\exp\left(\frac{2\kappa}{4+\epsilon/2}\right)\right) \\
    &\geq\, \ell_1\exp\left(-\frac{4}{4+\epsilon/10}\,\exp\left(\frac{2\kappa}{4+\epsilon/2}\right)\,\log d\right) \\
    &=\, \ell_1\exp\left(-\frac{4(1+\epsilon/50)}{4+\epsilon/10}\,\log d\right) \\
    &\geq\, d^{a\,\epsilon}.
\end{align*}
This proves \ref{item:11}. To show \ref{item:12}, note the following:
\begin{align*}
r_{i} \,\leq\, \left(1 - \frac{(1-\epsilon/10)\kappa}{2\log d}\,\exp\left(-\frac{2\kappa}{4+\epsilon/2}\right)\right)^{i-1}r_1 \,\leq\, \exp\left(-\frac{(i-1)(1-\epsilon/10)\kappa}{2(1+\epsilon/50)\log d}\right)\,\frac{\log d}{(4+\epsilon)}.
\end{align*}
For $i = \frac{10}{\kappa}\,\log d\log\log d$, the right hand side is smaller than $1/100$.
\end{proof}

    Let $i^\star$ be the minimum integer satisfying statement \ref{item:12} of Lemma~\ref{lemma:i_star}. Note that for all $i < i^\star$, \[d_i \,\geq\, \ell_i/100 \,\geq\, d^{a'\epsilon},\] where $a' > 0$ is some absolute constant. 
    Let $1 \leq i' \leq i^\star-1$ be maximum such that
    \[
        \beta_{i'} \,=\, d_{i'}^{-1/(200t)}.
    \]
    Then, for large enough $d$ and small enough $\alpha$, we can write
    \begin{align*}
        \beta_{i^\star-1} \,=\, (1+36\eta)^{i^\star - i' - 1} \beta_{i'} \,&\leq\, (1+36\eta)^{i^\star}d_{i'}^{-1/(200t)} \\
        &\leq\,\exp\left(36\eta\,i^\star\right)\,d^{-a'\epsilon/(200t)} \,\leq\, \frac{1}{\log d} \,\leq\, 1/10.
    \end{align*}
This verifies \ref{item:eps} for all $i < i^\star$.
By Lemma \ref{iterationTheorem} and the above upper bound on $\beta_{i^\star - 1}$, we have
\begin{align*}
    \min_{v \in V(G_{i^\star})} |L_{i^\star}(v)| \,\geq\, (1-(1+36\eta)/\log d)\,\ell_{i^\star}/2 \,\geq\,  50(1-2/\log d)\,d_{i^\star} \,\geq\, 10\max\limits_{c\in V(H_{i^\star})}\deg_{H_{i^\star}}(c).
\end{align*}
Thus $G_{i^\star}$ and $\mathcal{H}_{i^\star}$ satisfy the assumption of Proposition \ref{finalBlow} and we can complete the coloring.

It remains to verify statements \ref{item:1}--\ref{item:5} for $i < i^\star$. Notice that if we take $d_0 > \tilde{d}^{1/(a'\epsilon)}$, then $d_i \geq \tilde{d}$ for all $i < i^\star$. Hence, \ref{item:1} is satisfied. Moreover, since $a'\epsilon < 1$, we have
\[d_i^{1/10} \geq d^{a'\epsilon/10} \qquad \text{and} \qquad \tilde{\alpha}\frac{\log d_i}{\log \log d_i} \geq \tilde{\alpha}\frac{a'\epsilon\log d}{(\log \log d +\log (a\epsilon))} \geq \tilde{\alpha}\frac{a'\epsilon\log d}{\log \log d},\]
so, for $\alpha \leq \min\{a'\tilde{\alpha},a'/10\}$, condition \ref{item:3} is satisfied as well.

It remains to show that conditions \ref{item:2} and \ref{item:5} are satisfied at every iteration. The bound $\ell_i < 100\,d_i$ holds by the choice of $i^\star$. The lower bound on $\ell_i$ is proved as follows:
\[\frac{\ell_i}{d_i} \,\geq\, \frac{\ell_1}{d_1} \,=\, \frac{4+\epsilon}{\log d} \,\geq\, 4\eta.\]
Finally, it follows for $d$ large enough that:
\[\frac{1}{\log^5d_i} \,\leq\, \frac{1}{(a'\epsilon)^5\log^5 d} \,\leq\, \eta \,\leq\, \frac{1}{\log d} \,\leq\, \frac{1}{\log d_i}.\]
As discussed earlier, we can now iteratively apply Lemma~\ref{iterationTheorem} $i^\star - 1$ times and then complete the coloring using Proposition~\ref{finalBlow}. This finishes the proof of Theorem~\ref{mainTheorem}.

    \section{Palette sparsification}\label{section:sparse}
    
    In this section we derive Corollary~\ref{corl:sparse} from Corollary~\ref{mainCorollary}. As mentioned in \S\ref{subsec:alg}, this is a variant of the argument used by Alon and Assadi in \cite[\S3.2]{Palette} to prove Theorem~\ref{sparseAA}. To begin with, we need a probabilistic tool, namely a version of the Chernoff bound for negatively correlated random variables introduced by Panconesi and Srinivasan \cite{panconesi}. We say that $\set{0,1}$-valued random variables $X_1$, \ldots, $X_m$ are \emphd{negatively correlated} if for all $I \subseteq [m]$,
    \[
        \E\left[\prod_{i \in I} X_i\right] \,\leq\, \prod_{i \in I} \E[X_i].
    \]
    
    \begin{theo}[{\cite[Theorem 3.2]{panconesi}, \cite[Lemma 3]{Molloy}}]\label{lemma:NegativeChernoff}
    Let $X_1$, \ldots, $X_m$ be $\set{0,1}$-valued random variables. Set $X \defeq \sum_{i=1}^mX_i$. If $X_1$, \ldots, $X_m$ are negatively correlated, then, for all $0 < t \leq \E[X]$,
    \[
        \P[X > \E[X] + t] \,<\, \exp{\left(-\frac{t^2}{2\E{[X]}}\right)}.
    \]
    \end{theo}
    
    Now we fix $\epsilon > 0$ and  an almost bipartite graph $F$. Let
    \[
        C \,\defeq\, 300 \epsilon^{-3/2}.
    \]
    Take any $0 < \gamma < 1$ and let $d$, $n$, $s$ and $\ell$ satisfy the assumptions of Corollary~\ref{corl:sparse}. That is, we assume that $d$ is sufficiently large as a function of $\epsilon$, $F$, and $\gamma$, and we have
    \[
            \ell \,=\, \frac{4+\epsilon}{\gamma} \,\frac{d}{\log d} \qquad \text{and} \qquad s \,\geq\, d^\gamma + C\sqrt{\log n}.
    \]
    Since $d$ is large as a function of $\epsilon$ and $\gamma$, we may assume that $d \geq \ell$. Let $\mathcal{H} = (L,H)$ be a DP-cover of an $n$-vertex graph $G$ such that:
        \begin{enumerate}[label=\ep{\normalfont\roman*}]
            \item $H$ is $F$-free,
            \item $\Delta(H) \leq d$, and
            \item $|L(v)| = \ell$ for all $v \in V(G)$.
        \end{enumerate}
      Independently for each vertex $v \in V(G)$, pick a uniformly random subset $S(v) \subseteq L(v)$ of size $s$. Let $S \defeq \bigcup_{v \in V(G)}S(v)$ be the set of all picked colors. Set $\epsilon' \defeq \epsilon/20$ and define
      \[
        S'(v) \,\defeq\, \set{c \in S(v) \,:\, |N_H(c) \cap S| \leq (1+\epsilon') sd/\ell}.
      \]
      
      \begin{Lemma}\label{lemma:Sprime}
        With probability at least $1 - 1/n$, $|S'(v)| \geq (1-\epsilon')s$ for all $v \in V(G)$.
      \end{Lemma}
      \begin{proof}
        Take any vertex $v \in V(G)$. We will prove that
        \[
            \P\left[|S'(v)| < (1-\epsilon')s\right] \,\leq\, n^{-2},
        \]
        which gives the desired result by the union bound. Let us start by sampling $S(v)$. Now we assume that $S(v)$ is fixed. Let $t \defeq \lceil \epsilon' s \rceil$. Call a subset $T \subseteq S(v)$ of size $t$ \emphd{bad} if
        \[
            |N_H(T) \cap S| \,>\, (1+\epsilon')\frac{dts}{\ell}.
        \]
        Observe that if $|S'(v)| < (1-\epsilon')s$, then $|S(v) \setminus S'(v)| \geq t$ and every $t$-element subset of $S(v) \setminus S'(v)$ is bad. Therefore, it suffices to argue that
        \[
            \P[\text{there is a bad set $T \subseteq S(v)$ of size $t$}] \, \leq\, n^{-2}.
        \]
        To this end, consider an arbitrary set $T \subseteq S(v)$ of size $t$. For each color $c \in N_H(T)$, let $X_c$ be the indicator random variable of the event that $c \in S$. Define
        \[
            X \,\defeq\, \sum_{c \in N_H(T)} X_c \,=\, |N_H(T) \cap S|.
        \]
        Then $\E[X] \leq dts/\ell$, because $\E[X_c] = s/\ell$ for all $c \in N_H(T)$ and $|N_H(T)| \leq \Delta(H) |T| \leq dt$. Next we note that the random variables $(X_c \,:\, c \in N_H(T))$ are negatively correlated. Indeed, take any $I \subseteq N_H(T)$ and $c \in N_H(T) \setminus I$. Let $u \defeq L^{-1}(c)$ be the underlying vertex of $c$. Then 
        \[
            \P\left[X_c = 1 \,\middle\vert\, \prod_{c' \in I} X_{c'} = 1\right] \,=\, \frac{s- |I \cap L(u)|}{\ell - |I \cap L(u)|} \,\leq\, \frac{s}{\ell}.
        \]
        Inductive applications of this inequality show that for any $I \subseteq N_H(T)$,
        \[
            \E\left[\prod_{c \in I} X_c\right] \,\leq\, \left(\frac{s}{\ell}\right)^{|I|} \,=\, \prod_{c \in I} \E[X_c],
        \]
        as desired. Using Theorem~\ref{lemma:NegativeChernoff}, we conclude that
        \[
            \P\left[X > (1+\epsilon')\frac{d t s}{\ell}\right] \,<\, \exp\left(- \frac{(\epsilon')^2}{3} \, \frac{dts}{\ell}\right) \,\leq\, \exp\left(- \frac{(\epsilon')^3}{3} \, s^2\right).
        \]
        By the union bound, it follows that
        \begin{align*}
            \P\left[\text{there is a bad set $T \subseteq S(v)$ of size $t$}\right] \,\leq\, 2^s \, \exp\left(- \frac{(\epsilon')^3}{3} \, s^2\right) \,\leq\, \exp\left(- \frac{(\epsilon')^3}{4} \, s^2\right),
        \end{align*}
        where the second inequality holds if $d$ is large enough. Finally, since $s^2 \geq C^2\log n$, we have
        \[
            \exp\left(- \frac{(\epsilon')^3}{4} \, s^2\right) \,\leq\, \exp\left(- \frac{(\epsilon')^3}{4} \, C^2 \log n\right) \,\leq\, n^{-2},
        \]
        as desired.
      \end{proof}
      
      Let $S' \defeq \bigcup_{v \in V(G)}S'(v)$ and $H' \defeq H[S']$. Then $\mathcal{H}' \defeq (S', H')$ is a DP-cover of $G$ such that:
      \begin{enumerate}[label=\ep{\normalfont\roman*}]
            \item $H'$ is $F$-free,
            \item $\Delta(H') \leq (1+\epsilon')sd/\ell$, and
            \item with probability at least $1 - 1/n$, $|S'(v)| \geq (1-\epsilon')s$ for all $v \in V(G)$.
        \end{enumerate}
    Note that, assuming $d$ is large enough, $(1+\epsilon')sd/\ell \geq s \geq d^\gamma$, so
    \[
        (4 + \epsilon') \, \frac{(1+\epsilon')sd/\ell}{\log \left((1+\epsilon')sd/\ell\right)} \,\leq\, (4 + 6\epsilon') \,\frac{sd/\ell}{\gamma \log d} \,=\, \frac{4 + 6\epsilon'}{4 + \epsilon} \, s \,<\, (1-\epsilon')s.
    \]
    Therefore, with probability at least $1 - 1/n$, $G$ is $\mathcal{H}'$-colorable by Corollary~\ref{mainCorollary} \ep{applied with $\epsilon'$ in place of $\epsilon$}, and the proof of Corollary~\ref{corl:sparse} is complete.
    
    \subsubsection*{Acknowledgments}
    
    We are very grateful to Sepehr Assadi for helpful comments on an earlier version of this manuscript, in particular for drawing our attention to the consequences of our results for palette sparsification discussed in \S\S\ref{subsec:alg} and \ref{section:sparse}. We are also grateful to the anonymous referees for carefully reading the manuscript and providing helpful comments and suggestions.

    \frenchspacing

\printbibliography

\end{document}